\documentclass[11pt,a4paper]{article}%
\usepackage{amssymb,amsmath, amsfonts}
\usepackage{graphicx,graphics}
\usepackage{mathtools}
\usepackage{bbold}
\usepackage[english]{babel}
\usepackage[utf8]{inputenc}
\usepackage{epsfig,url}
\usepackage{bbm,theorem}
\usepackage{a4wide}
\usepackage{color}
\usepackage{enumerate}
\usepackage{calrsfs}
\usepackage[bookmarks=true]{hyperref}
\usepackage{bookmark}
\usepackage{amsmath}
\usepackage{amsfonts}
\usepackage{amssymb}
\usepackage{verbatim}
\usepackage{graphicx}%
\providecommand{\U}[1]{\protect\rule{.1in}{.1in}}
\providecommand{\U}[1]{\protect\rule{.1in}{.1in}}
\DeclareMathAlphabet{\pazocal}{OMS}{zplm}{m}{n}
\newtheorem{theorem}{Theorem}[section]
\newtheorem{definition}[theorem]{Definition}

\newtheorem{lemma}[theorem]{Lemma}
\newtheorem{proposition}[theorem]{Proposition}
\newtheorem{assumption}[theorem]{Assumption}

{\theorembodyfont{\upshape}
\newtheorem{remark}[theorem]{Remark}

}
\numberwithin{equation}{section}
\numberwithin{theorem}{section}
\newcommand{\qed}{\hfill$\Box$}
\newenvironment{proof}{\begin{trivlist}\item[]{\em Proof:}\/}{\qed\end{trivlist}}
\newenvironment{proofof}[1][Proof]{\noindent \textit{#1.} }{\ \qed}

\newcommand{\R}{{\mathbb R}}
\newcommand{\C}{{\mathbb C\hspace{0.05 ex}}}
\newcommand{\N}{{\mathbb N}}

\newcommand{\ep}{\varepsilon}

\newcommand{\beq}{\begin{equation}}
\newcommand{\eeq}{\end{equation}}
\newcommand{\beqs}{\begin{eqnarray}}
\newcommand{\eeqs}{\end{eqnarray}}

\newcommand{\defset}[2]{ \left\{ #1\,{:}\,
#2\makebox[0cm]{$\displaystyle\phantom{#1}$}\right\} }

\newcommand{\vep}{\varepsilon}
\newcommand{\defem}[1]{{\em #1\/}}

\newcounter{jlisti}

\newcommand{\cf}[1]{{\mathbbm 1}_{\{#1\}}}

\begin{document}

\title{Multicomponent coagulation systems: existence and non-existence of stationary non-equilibrium solutions}
\author{Marina A. Ferreira, Jani Lukkarinen, Alessia Nota, Juan J. L. Vel\'azquez}
\maketitle

\begin{abstract}

We study multicomponent coagulation via the Smoluchowski coagulation equation
under non-equilibrium stationary conditions induced by a source of small 
clusters. 
The coagulation kernel can be very general, merely satisfying certain 
power law asymptotic bounds in terms of the total number of monomers in a cluster.
The bounds are characterized by two parameters and we extend previous results for one-component systems to classify 
the parameter values for which the above stationary solutions  do  or do not exist. 
Moreover, we also obtain criteria for the existence or non-existence of solutions which yield a constant flux of mass towards large clusters. 
\end{abstract}


\tableofcontents

\bigskip

\bigskip

\section{Introduction}

Smoluchowski's coagulation equation is a classical model for binary aggregation extensively used in the study of aerosol growth, polymerization, drop formation in rain and several other situations \cite{Fried, Vehkam, Vehkamki2012}. Particles undergo complex phenonena that influence their movement, size and composition. These particles grow due to coalescence, however, how they grow and how their composition changes is not well understood.  

The particle clusters are made of aggregates of
different types of coagulating molecules, which are called 
\defem{monomers}. 
We denote by $n_{\alpha}$  the
concentrations of multicomponent clusters, with composition $\alpha=\left(  \alpha_{1},\alpha
_{2},\ldots ,\alpha_{d}\right)  \in\mathbb{N}_{0}^{d}$ where ${\alpha}_i\in\mathbb{N}_{0} $ denotes the number of monomers
of type $i$. Notice that  $\mathbb{N}%
_{0}=\left\{  0,1,2,3,\ldots \right\}$ and we denote $O=\left(0,0,\ldots ,0\right)$.  The multicomponent Smoluchowski's coagulation equation, which describes  the evolution of the clusters concentrations $\{n_{\alpha}\}_{\alpha\in\mathbb{N}_{0}^{d}\backslash \{O\}}$, is given by 
\begin{equation}
\partial_{t}n_{\alpha}=\frac{1}{2}\sum_{\beta<\alpha}K_{\alpha-\beta,\beta
}n_{\alpha-\beta}n_{\beta}-n_{\alpha}\sum_{\beta>0}K_{\alpha,\beta}n_{\beta
}+s_{\alpha}\,,
\label{B2} 
\end{equation}
where $\alpha=\left(  \alpha_{1},\alpha_{2},\ldots ,\alpha_{d}\right)  $
and $\beta=\left(  \beta_{1},\beta_{2},\ldots ,\beta_{d}\right)  .$ The coefficients $K_{\alpha,\beta}$ describe the coagulation rate between
clusters with compositions $\alpha$ and $\beta$. 
 We use the
notation $\beta<\alpha$ to indicate that $\beta_{k}\leq\alpha_{k}$ for all
$k=1,2,\ldots ,d$, and in addition $\alpha\neq\beta.$ We denote as $s_{\alpha}$ the
source of small particles characterized by the composition $\alpha.$ We will allow  source terms 
$s_{\alpha}$ which are 
supported on a finite set of values $\alpha$.

The coefficients $K_{\alpha,\beta}$ yield the coagulation rate between
clusters $\alpha$ and $\beta$ to produce clusters $\left(  \alpha
+\beta\right)  .$ The form of these coefficients depends on the specific
mechanism which is responsible for the aggregation of the clusters.
These coefficients have been computed using the method of kinetic theory under different
assumptions on the particle sizes and the processes describing the motion of the
clusters.

Relevant examples of coagulation kernels have been described in the literature, see e.g.\ the textbook 
\cite{Fried}.  Typical ones are given by  the \emph{free molecular regime
coagulation kernel}  
\begin{equation}
K_{\alpha,\beta}= 
\left(  \frac{1}{V\left(  \alpha\right)  }+\frac
{1}{V\left(  \beta\right)  }\right)  ^{\frac{1}{2}}\left(  \left(  V\left(
\alpha\right)  \right)  ^{\frac{1}{3}}+\left(  V\left(  \beta\right)  \right)
^{\frac{1}{3}}\right)  ^{2}\,. \label{KerBall}%
\end{equation}
and the \emph{diffusive coagulation} or \emph{Brownian kernel}
\begin{equation}
K_{\alpha,\beta}= 
\left(  \frac{1}{\left(  V\left(
\alpha\right)  \right)  ^{\frac{1}{3}}}+\frac{1}{\left(  V\left(
\beta\right)  \right)  ^{\frac{1}{3}}}\right)  \left(  \left(  V\left(
\alpha\right)  \right)  ^{\frac{1}{3}}+\left(  V\left(  \beta\right)  \right)
^{\frac{1}{3}}\right).  \label{KernBrow}
\end{equation}
Here $V\left(  \alpha\right)$  is the volume of the cluster characterized by the composition $\alpha$. More details on the physical properties and the derivation of the kernels can be found in \cite{FLNV}, \cite{FLNV1}.   
In these formulas we will assume that the volume scales linearly with the number of monomers in the cluster. More precisely, 
\begin{equation}
k_{1}\left\vert \alpha\right\vert \leq V\left(  \alpha\right)  \leq
k_{2}\left\vert \alpha\right\vert \ \ \ \ \text{with}\quad0<k_{1}\leq
k_{2}<\infty\,, \label{MasVolIneq}
\end{equation}
where $|\cdot|$ denotes the $\ell^1$-norm, i.e.,
\begin{equation}
\label{eq:l1norm}\left\vert \alpha\right\vert =\sum_{j=1}^{d}\alpha_{j}, \quad \alpha \in\mathbb{N}_{0}^{d}\setminus\{O\}.
\end{equation}
 The inequalities (\ref{MasVolIneq}) hold, for
instance,  if we assume $V\left(  \alpha\right)  =\sum_{j=1}^{d}\alpha_{j}v_{j}$ where
$v_{j}>0$ represents the volume of the monomer of type $j$ for each $j=1,2,\ldots ,d.$

\medskip

We also consider the continuous version of (\ref{B2}) which is given by 
\begin{align}
& \partial_{t}f\left(  x\right)  =\frac{1}{2}\int_{\left\{  0<\xi<x \right\}
}d\xi K\left(  x-\xi,\xi\right)  f\left(  x-\xi\right)  f\left(  \xi\right)
\nonumber \\ & \qquad
-\int_{ \mathbb{R}_{+}^{d}\setminus\{O\}}d\xi K\left(  x,\xi\right)  f\left(  x\right)
f\left(  \xi\right)  +\eta\left(  x\right) \, ,\quad x\in \mathbb{R}_{+}^{d}\setminus\{O\} \label{B3}%
\end{align}
where $f$ denotes the density of clusters with composition $x\in \mathbb{R}_{+}^{d}\setminus\{O\} $. In the same way as in the discrete case, given $x=\left(  x_{1},x_{2},\ldots ,x_{d}\right)  ,$ $y=\left(  y_{1}%
,y_{2},\ldots ,y_{d}\right)  $ we would say that $x<y$ whenever $x\le y$ componentwise, and $x\ne y$.
In particular, 
\[
\int_{\left\{  0<\xi<x \right\}  } d\xi=\int_{0}^{x_{1}}d\xi_{1}\int
_{0}^{x_{2}}d\xi_{2}\cdots \int_{0}^{x_{d}}d\xi_{d}\,.
\]

We notice that the discrete model \eqref{B2} can be thought as a particular case of the continuous one if we assume that  $f$ is the sum of Dirac measures  supported at points with integer coordinates. 

\medskip

In this paper we will consider the stationary solutions to the problems
(\ref{B2}) and (\ref{B3}). In order to obtain nontrivial solutions we will
require that $\sum_{\beta}s_{\beta}>0$ in
(\ref{B2}). In the case of (\ref{B3}) we will assume that $\eta$ is a  Radon measure with $0<\int \eta\left(  dx\right) <\infty$.

We will restrict our attention to the class of coagulation kernels satisfying 
\begin{align}\label{Ineq1}
c_{1}\left(  \left\vert \alpha\right\vert +\left\vert \beta\right\vert
\right)  ^{\gamma}\Phi\left(  \frac{\left\vert \alpha\right\vert }{\left\vert
\alpha\right\vert +\left\vert \beta\right\vert }\right)   &  \leq
K_{\alpha,\beta}\leq c_{2}\left(  \left\vert \alpha\right\vert +\left\vert
\beta\right\vert \right)  ^{\gamma}\Phi\left(  \frac{\left\vert \alpha
\right\vert }{\left\vert \alpha\right\vert +\left\vert \beta\right\vert
}\right)   ,\ \alpha,\beta\in\mathbb{N}_{0}^{d}\setminus\{O\}
\\
c_{1}\left(  \left\vert x\right\vert +\left\vert y\right\vert \right)
^{\gamma}\Phi\left(  \frac{\left\vert x\right\vert }{\left\vert x\right\vert
+\left\vert y\right\vert }\right)   &  \leq K\left(  x,y\right)  \leq
c_{2}\left(  \left\vert x\right\vert +\left\vert y\right\vert \right)
^{\gamma}\Phi\left(  \frac{\left\vert x\right\vert }{\left\vert x\right\vert
+\left\vert y\right\vert }\right)  ,\, x,y\in\mathbb{R}_{+}^{d}\setminus\{O\}
\label{Ineq2}
\end{align}
where 
$\left\vert \cdot \right\vert $ denotes the $\ell^{1}$ norm as in \eqref{eq:l1norm} and 
\begin{equation}
\Phi\left(  s\right)  =\Phi\left(  1-s\right)  \text{ for }%
0<s<1\ \ ,\ \ \ \Phi\left(  s\right)  =\frac{1}{s^{p}\left(  1-s\right)  ^{p}%
}\ \ \text{for \ }0<s<1\ \ ,\ \ p\in\mathbb{R} \,,\label{Ineq3}%
\end{equation}
where $0<c_{1}\leq c_{2}<\infty.$   
This class of kernels includes
the physically relevant kernels (\ref{KerBall}) and (\ref{KernBrow}).  We stress that even though the estimates \eqref{Ineq1}, \eqref{Ineq2} are isotropic, i.e., invariant under permutation of components, the kernels are not necessarily isotropic. 

The existence of steady states to the problems (\ref{B2}),\ (\ref{B3}) in the
case $d=1$ has been considered in \cite{FLNV} for a   less general class of kernels than the one covered by the assumptions 
(\ref{Ineq1})--(\ref{Ineq3}). Indeed, the conditions assumed in \cite{FLNV} (with $d=1$) are
\begin{align}
c_{1}\left(  \alpha^{\gamma+\lambda}\beta^{-\lambda}+\beta^{\gamma+\lambda
}\alpha^{-\lambda}\right)   &  \leq K_{\alpha,\beta}\leq c_{2}\left(
\alpha^{\gamma+\lambda}\beta^{-\lambda}+\beta^{\gamma+\lambda}\alpha
^{-\lambda}\right)\,, \label{PrevIned1}\\
c_{1}\left(  x^{\gamma+\lambda}y^{-\lambda}+y^{\gamma+\lambda}x^{-\lambda
}\right)   &  \leq K\left(  x,y\right)  \leq c_{2}\left(  x^{\gamma+\lambda
}y^{-\lambda}+y^{\gamma+\lambda}x^{-\lambda}\right)  \label{PrevIneq2}%
\end{align}
for some $\gamma,\lambda\in\mathbb{R}$. It is readily seen that the kernels
satisfying (\ref{PrevIned1}), (\ref{PrevIneq2}) satisfy also (\ref{Ineq1}%
)--(\ref{Ineq3}) with $p=\max\left\{  \lambda,-\left(  \gamma+\lambda\right)
\right\}  $ (assuming $d=1$). On the other hand, for
any $p\in\mathbb{R}$ with $p\geq-\frac{\gamma}{2}$ there exists at least one
value $\lambda\in\mathbb{R}$ such that $\max\left\{  \lambda,-\left(
\gamma+\lambda\right)  \right\}  =p$.  In fact, we can take by definiteness
$\lambda=p,$ since then $-\left(  \gamma+\lambda\right)
\leq2p-\lambda=\lambda,$ and therefore, $\max\left\{  \lambda,-\left(  \gamma
+\lambda\right)  \right\}  =\lambda=p.$ If $p<-\frac{\gamma}{2}$, it is not
possible to choose $\lambda$ such that $p=\max\left\{  \lambda,-\left(
\gamma+\lambda\right)  \right\}  .$ Therefore, the class of kernels satisfying
(\ref{Ineq1})--(\ref{Ineq3}) is strictly larger than the class satisfying
(\ref{PrevIned1}), (\ref{PrevIneq2}).

 In this paper we will prove that in the
multicomponent case and under the assumptions (\ref{Ineq1})--(\ref{Ineq3}),
there exists a stationary solution to (\ref{B2}), (\ref{B3}) if and only if 
\begin{equation}
\gamma+2p<1\, . \label{ExCond}%
\end{equation}
We note that for $p=\max\left\{  \lambda,-\left(  \gamma
+\lambda\right)  \right\} $, condition \eqref{ExCond} implies the condition $|\gamma+2\lambda|<1$ obtained in \cite{FLNV}  for the kernels satisfying (\ref{PrevIned1}), (\ref{PrevIneq2}). 
Indeed, if
$\lambda\geq-\left(  \gamma+\lambda\right)  $, we have $\gamma+2\lambda\geq0$
and, since $p=\lambda$, (\ref{ExCond}) is equivalent to $\gamma+2\lambda<1.$ Otherwise, if
$\lambda<-\left(  \gamma+\lambda\right)  $, we have $\gamma
+2\lambda<0,\ p=-\left(  \gamma+\lambda\right)  ,$ and thus (\ref{ExCond})
is equivalent to $\gamma+2\lambda>-1.$ Therefore, (\ref{ExCond}) holds if
and only if $\left\vert \gamma+2\lambda\right\vert <1$, whenever the two cases can be compared.

Notice that these steady states yield a transfer of monomers from small clusters to large clusters in the space of clusters sizes. Their existence express the balance between the injection of small clusters (e.g. monomers) and the transport of these monomers towards clusters of infinite size due to the coagulation mechanism. The non-existence of these steady states is due to the fact that the transport of monomers towards large clusters is too fast and cannot be balanced by any monomers injection, and therefore no stationary regime is possible.
We emphasize that the steady states of (\ref{B2}), (\ref{B3}) are stationary non-equilibrium solutions for an open system.

It is worth to mention that in the case of discrete kernels with the form
$K_{\alpha,\beta}=\alpha^{\gamma+\lambda} \beta^{-\lambda}+ \alpha^{-\lambda} \beta^{\gamma+\lambda}$,
and source terms supported at the monomers, the stationary solutions of
\eqref{B2} in dimension $d=1$ have been computed explicitly in \cite{HH} assuming the non gelling condition $\max\{\gamma+\lambda, -\lambda\}\leq 1$. It turns out that these solutions  are well defined, non-negative, densities of clusters distributions if and only if  
$\vert \gamma+2\lambda\vert<1$ holds. 

It is interesting to note that the existence or nonexistence of stationary
solutions to (\ref{B2}), (\ref{B3}) is independent of the number of components
$d$ here.  Both cases are also already represented by the two example kernels discussed above. In the case of kernels with the form (\ref{KerBall}), we
have $\gamma=\frac{1}{6}$ and $p=\frac{1}{2}$.  Thus the inequality
(\ref{ExCond}) is not satisfied, and there are no stationary solutions.
On the other hand, in the case of kernels
with the form (\ref{KernBrow}) we have $\gamma=0$ and $p=\frac{1}{3}$.  Then
the inequality (\ref{ExCond}) holds, and there exists at least one stationary solution.

In this paper we will prove the existence of steady states to \eqref{B2}, \eqref{B3} under the
assumption \eqref{ExCond} and nonexistence of steady states if $\gamma
+2p\geq1$. The proofs of these results require a generalization of the methods developed in \cite{FLNV} to the multicomponent case. 
In particular, in the multicomponent setting,  the mass flows from small to large sizes through a (d-1)-dimensional surface rather than a point, as in the one-dimensional case. This allows for more possibilities in the choice of the  definition of flux and some care is needed in choosing an appropriate definition (cf.\ Section \ref{ss:stattypes}). Moreover, in the multicomponent setting, the proofs require more refined geometrical arguments than the ones used in the one-component case.

In the physical literature, explicit stationary solutions to the multicomponent equation \eqref{B2} have been obtained in \cite{KbN} in the case of the constant kernel $K(x,y)=1$ and additive kernel $K(x,y)=x+y $ and with a source term supported on the monomers. 

For non-solving kernels, most of the mathematical analysis of coagulation equations has been made for
one-component systems only, i.e., $d=1$.  On the other hand, there are only a few papers
addressing the problem of the coagulation equations with injection terms like
$\{s_{\alpha}\}_\alpha$ or
$\eta$ (cf. \cite{Dub, EM06, FLNV, Laurencotpreprint20}). This issue has been discussed in \cite{FLNV} and we refer to that
paper for additional references.

An interesting property of the steady
states to (\ref{B2}), (\ref{B3}) specific to the multicomponent
coagulation system, that does not have a counterpart in the case $d=1$,
is the so-called 
\textit{localization property}.  It consists in the fact that the concentrations $n_{\alpha}$ localize along a particular line in the space $\mathbb{N}_{0}^{d}$ as $\left\vert \alpha\right\vert \to \infty$. A similar property holds in the continuous case, namely the density $f$ concentrates along a specific direction of the cone ${\mathbb{R}}^{d}_{+}$ as $\left\vert x\right\vert \to \infty$. The precise formulation is the following. If $n_{\alpha}$ and $f$ are  stationary solutions to  (\ref{B2}), (\ref{B3}) respectively then there is a $\zeta>0$ such that, for any $\ep>0$, 
\begin{equation}
\lim_{R\rightarrow\infty}\frac{\sum_{\left\{  R\leq\left\vert \alpha
\right\vert \leq \zeta R\right\}  \cap\left\{  \left\vert \frac{\alpha}{\left\vert
\alpha\right\vert }-\theta\right\vert <\varepsilon\right\}  }n_{\alpha}}%
{\sum_{\left\{  R\leq\left\vert \alpha\right\vert \leq\zeta R\right\}  }n_{\alpha}%
}=1\ \ \text{or\ \ }\lim_{R\rightarrow\infty}\frac{\int_{\left\{
R\leq\left\vert x\right\vert \leq\zeta R\right\}  \cap\left\{  \left\vert \frac
{x}{\left\vert x\right\vert }-\theta\right\vert <\varepsilon\right\}
}f\left(  dx\right)  }{\int_{\left\{  R\leq\left\vert x\right\vert
\leq\zeta R\right\}  }f\left(  dx\right)  }=1. \label{LocStat}%
\end{equation}
where the direction $\theta$ is defined by the normalized mass vector of the source $s_\alpha$ or $\eta$ such that $|\theta|=1$. The proof of this result is given in \cite{FLNV1} for the class of kernels satisfying (\ref{Ineq1})--(\ref{Ineq3}).

\subsubsection*{Structure of the paper }

The plan of the paper is the following. In Section \ref{ss:stattypes} we
informally discuss the different types of stationary solutions considered in
this paper (constant injection solutions, constant flux solutions, \dots). In
Section \ref{StatDef} we introduce rigorously the definitions of solutions
studied in this paper. In Section \ref{sec:4} we formulate 
the main results that we prove in this paper, namely, existence or nonexistence of stationary injection solutions or constant flux solutions for several classes of kernels. Section
\ref{sec:5} contains two technical results which are repeatedly used in the
rest of the paper. The proof of the existence of  steady states for some classes of kernels is the
content of Section \ref{sec:6}. The non-existence results for a different class of kernels are given in Section \ref{sec:9}.
Section \ref{sec:7} provides some estimates
for the stationary solutions, whenever they exist.

\subsubsection*{Notations}

We will denote by $\mathbb{R}_{+}:={[}0,\infty {)} $ and $\N_0:= \{0,1,2,\ldots\}$ the non-negative real numbers and integers respectively.  We also 
use a subindex ``$*$'' to denote restriction of real-component vectors $x$ to those which satisfy $x>0$,  or equivalently $\max_i (x_i)>0$.  More precisely, we denote
$\R_*:=\R_+\setminus\{0\}$, $\R^d_*:=\R_+^d\setminus\{O\}$ and $\N^d_*:=\N_0^d\setminus\{O\}$.
Given a locally compact Hausdorff space $X$ (for instance $X=\R^d_*$) we denote by  
$C_{c}\left(X\right) $ the space of compactly supported
continuous functions from $X$ to $\C$, and by $C_0(X)$  its completion in the standard supremum norm. 
The collection of non-negative Radon measures on $X$, not necessarily
bounded, will be denoted by $\mathcal{M}_{+}(X)$ and its subspace consisting of bounded
measures by $\mathcal{M}_{+,b}(X)$.  Due to the Riesz--Markov--Kakutani theorem, we can identify $ \mathcal{M}_{+}(X) $  with the space of positive linear functionals on $C_{c}(X)$.

Both the notation $\eta(x) dx$ and $\eta(dx)$ will be used to denote elements of the above measure spaces.
We will use the symbol $\eta( dx) $ when performing integrations or when we want to emphasize that the measure might not be absolutely continuous with respect to the Lebesgue measure.  We will often drop the differential ``$dx$'' from the first notation, typically when the measure eventually turns out to be absolutely continuous. 
We will use the symbol $\cf{P}$ to denote the characteristic function of a condition $P$: $\cf{P}=1$ if the condition $P$ is
true, and $\cf{P}=0$ if $P$ is false.

%%%%%%%%%%%%%%%%%%%%%%%%%%%%

\section{Different types of stationary solutions for multicomponent coagulation
equations\label{StatTypes}}

\bigskip

We now introduce different types of stationary solutions of (\ref{B2}),
(\ref{B3}) which will be considered in this paper. These classes of solutions
have been discussed in \cite{FLNV} in the case $d=1.$ We will adapt the
definitions used in that paper to the multicomponent case. We recall that in
all the cases discussed in this Section, the solutions are stationary,
nonequilibrium solutions yielding a constant flux of monomers towards large
clusters. We discuss shortly here these classes of solutions as well as their
physical meaning.

\subsection{Heuristic description of flux and constant flux solutions}\label{ss:stattypes}

In this section, we first introduce different concepts of stationary solutions used in this paper.  The rigorous, more detailed, definitions are collected in Subsection \ref{StatDef}.

\subsubsection*{Stationary injection solutions}

The stationary solutions of (\ref{B2}), (\ref{B3}) satisfy respectively the equations 
\begin{align}
&0=\frac{1}{2}\sum_{\beta<\alpha}K_{\alpha-\beta,\beta}n_{\alpha-\beta}%
n_{\beta}-n_{\alpha}\sum_{\beta>0}K_{\alpha,\beta}n_{\beta}+s_{\alpha}\ \ ,\ \ \alpha\in\mathbb{N}_{*}^{d} \,,\label{B2Stat} \\
& 0=\frac{1}{2} \int_{\left\{  0<y <x \right\}  } K\left(  x-y,y\right)
f\left(  x-y,t\right)  f\left(  y,t\right)  dy-\int_{\mathbb{R}_{*}^{d}}K\left(
x,y\right)  f\left(  x,t\right)  f\left(  y,t\right)  dy
\nonumber \\ & \qquad 
+\eta\left(  x\right)
\,, \quad x\in\mathbb{R}_{*}^{d}\,.
\label{eq:contStat}%
\end{align}
We will assume that the sequence $s_{\alpha}$ is supported in a finite set of
values of $\alpha.$ On the other hand, we will assume that $\eta\in
\mathcal{M}_{+}\left(  \mathbb{R}_{*}^{d}\right)  $ is a Radon measure
compactly supported in the set $x\geq\mathbf{1}$, where $\mathbf{1}=\left(
1,1,\ldots ,1\right)  \in\mathbb{R}_{*}^{d}$
(for examples of how to relax the assumptions about the source, we refer to a recent preprint \cite{Laurencotpreprint20} where compact support is not required assuming that the solution $f$ is absolutely continuous with respect to the Lebesgue measure).
In this paper we are mostly interested in the solutions of the equations
(\ref{B2Stat}), (\ref{eq:contStat}) which we call \defem{stationary injection solutions}.  
Their detailed definition will be given in
Section \ref{StatDef}.

\subsubsection*{Constant flux solutions}

In addition to the above injection solutions, 
in the one-component case ($d=1$) we have considered in
\cite{FLNV} a family of solutions of (\ref{eq:contStat}) with $\eta=0$ that
we have termed as \textit{constant flux solutions}. 
The terminology and motivation arise from the fact that the coagulation equation without a source, at least formally, conserves ``total mass'', the function $\int x f(x,t) dx$. 
This conservation law leads to a continuity equation, which may be written as
\[
 \partial_t (x f(x,t)) + \partial_{x}J(  x;f)=0\,,
\]
where the \defem{flux} can be defined by 
\[
J\left(  x;f\right)  =\int_{0}^{x}dy\int_{x-y}^{\infty}dz \,K\left(  y,z\right)
yf\left(  y\right)  f\left(  z\right)\,.
\]
In this case, we find that $f$ is a stationary solution if and only if for all $x>0$
\begin{equation}
\partial_{x}J\left(  x;f\right)  =0\,, \label{fluxStat}%
\end{equation}
i.e., if and only if the flux is constant in $x$.
Therefore, if there is $J_0\ge 0$ and a measure $f\in\mathcal{M}_{+}\left(  \mathbb{R}_{*}\right)  $
such that
\begin{equation}
J\left(  x;f\right)  =J_{0}\,, \qquad \text{for all }x>0\,,
\label{EqConFlSol}%
\end{equation}
we say that $f$ is a constant flux solution.  
We say that the solution has a \defem{non-trivial flux} if $J_0>0$.  In this case, clearly also $f\ne 0$.

In the above one-dimensional case, any sufficiently regular constant flux solution $f$
also has the property that
\begin{equation}
0=-\frac{1}{x}\partial_{x}J(  x;f)=\frac{1}{2}\int_{0}^{x}K\left(  x-y,y\right)  f\left(  x-y,t\right)
f\left(  y,t\right)  dy-\int_{0}^{\infty}K\left(  x,y\right)  f\left(
x,t\right)  f\left(  y,t\right)  dy\,. \label{StatZeroSource}%
\end{equation}
Comparing the result with \eqref{eq:contStat} shows that these
are stationary solutions to the original evolution equation without source,
albeit with a slightly non-standard physical interpretation as solutions with non-trivial source of ``particles'' located at $x=0$. Indeed, one practical use for 
the constant flux solutions comes from the observation 
that they can provide the asymptotics of stationary injection solutions.
It has been proven in \cite{FLNV} that the
stationary injection solutions both of the discrete and the continuous model
(cf. (\ref{B2Stat}), (\ref{eq:contStat})) behave for large values of $\alpha$
or $x$ as a constant flux solution. More precisely, rescaling $n_{\alpha}$ or
$f$ in a suitable manner we obtain some measures that converge for large
values to a measure which satisfies (\ref{EqConFlSol}). We refer to
\cite{FLNV} for the detailed results.

In the multicomponent
case without source, the total mass of each of the particle species is conserved, so there  are 
now $d$ mass continuity equations, as derived below.  In addition,
the analogue of \eqref{fluxStat} is a vectorial divergence equation.  
Therefore,
it is not possible to characterize the fluxes at a given point just by one number.
In order to define a suitable concept of constant flux solutions in the
multicomponent case we must take into account that, if $d>1$, we cannot expect
the solutions of (\ref{StatZeroSource}) to be uniquely characterized by the
flux of particles across all the surfaces $\left\{  \left\vert x\right\vert
=R\right\}  $ for arbitrary values of $R>0$, where the norm $|\cdot|$ is as in
\eqref{eq:l1norm}. Let us introduce the change of variables $x\rightarrow
\left(  \left\vert x\right\vert ,\theta\right)  $ where $\theta=\frac
{x}{\left\vert x\right\vert }\in\Delta^{d-1}$ where we denote by $\Delta
^{d-1}$ the simplex
\begin{equation}
\label{eq:simplex}\Delta^{d-1}=\left\{  \theta\in\mathbb{R}_{*}^{d} :\left\vert
\theta\right\vert =1\right\} .
\end{equation}
Then, the detailed distribution of the measure $f$ in the variable $\theta$ in
each surface $\left\{  \left\vert x\right\vert =R\right\}  $ must be obtained
from the generalization of equation (\ref{fluxStat}) to the multicomponent
case and it cannot be determined just from the values of the fluxes across
these surfaces.

We now rewrite  equation (\ref{eq:contStat}) in the form of divergences of
fluxes.  
To this end, we choose a component $j\in \{1,2,\ldots,d\}$ and multiply (\ref{eq:contStat}) by $x_j$.  Expanding in the first term $x_j=(x-\xi)_j+\xi_j$ and using the symmetry $\xi\leftrightarrow\left(x-\xi\right)  $, we obtain 
\[
0=\int_{\left\{ 0<\xi<x\right\} }d\xi K\left(  x-\xi,\xi\right)  
\left(x-\xi\right)_j  f\left(  x-\xi\right)  f\left(  \xi\right)  -\int_{{\mathbb{R}}_{*}^{d}} d\xi K\left(  x,\xi\right)  x_j f\left(  x\right)  f\left(
\xi\right)  +x_j\eta\left(  x\right)\,.
\]
We then multiply this equation by a test function $\varphi\in  C_c(\R^d_*)$.
The support of $\varphi$ is a compact subset of $\R^d_+\setminus\{O\}$ and thus it is bounded and separated by a finite distance from the origin.  Thus we can find
$a,b>0$ with $a<b$ such that the support of $\varphi$ is  
contained in the set $\left\{  x:\left\vert x\right\vert \in\left[  a,b\right]  \right\}$. Then,
using Fubini's Theorem and assuming that all the integrals appearing in the computations
are finite, we obtain
\begin{equation}
0=\int_{\mathbb{R}_{*}^{d}}d\xi\int_{\mathbb{R}_{*}^{d}}dx\left[
\varphi\left(  x+\xi\right)  -\varphi\left(  x\right)  \right]  K\left(
x,\xi\right)  x_j f\left(  x\right)  f\left(  \xi\right)  +\int_{{\mathbb{R}}%
_{*}^{d}}x_j\eta\left(  x\right)  \varphi\left(  x\right)  dx \label{B8}%
\end{equation}
Here
\[
\varphi\left(  x+\xi\right)  -\varphi\left(  x\right)  =\int_{0}^{1}\xi
\cdot\nabla_{x}\varphi\left(  x+t\xi\right)  dt
\]
and thus
\[
0=\int_{0}^{1}dt\int_{\mathbb{R}_{*}^{d}}d\xi\int_{\mathbb{R}_{*}^{d}%
}dx\left[  \xi\cdot\nabla_{x}\varphi\left(  x+t\xi\right)  \right]  K\left(
x,\xi\right)  x_j f\left(  x\right)  f\left(  \xi\right)  +\int_{{\mathbb{R}}%
_{*}^{d}}x_j \eta\left(  x\right)  \varphi\left(  x\right)  dx
\]
Using the change of variables $y=x+t\xi$ in the first integral we obtain 
\begin{align*}
 & 0=\int_{0}^{1}dt\int_{\mathbb{R}_{*}^{d}}d\xi\int_{\left\{ t\xi<y\right\}
}dy\left[  \xi\cdot\nabla_{y}\varphi\left(  y\right)  \right]  K\left(
y-t\xi,\xi\right)  \left(  y-t\xi\right)_j  f\left(  y-t\xi\right)  f\left(
\xi\right) 
\\ & \quad
+\int_{\mathbb{R}_{*}^{d}}x_j\eta\left(  x\right)  \varphi\left(
x\right)  dx\,.
\end{align*}
Applying Fubini's Theorem we obtain 
\begin{align*}
 & 
0=\int_{\mathbb{R}_{*}^{d}}dy\nabla_{y}\varphi\left(  y\right)  \cdot\left[
\int_{0}^{1}dt\int_{\left\{ 0<\xi<\frac{y}{t}\right\} }d\xi\ \xi \left(
y-t\xi\right)_j  K\left(  y-t\xi,\xi\right)  f\left(  y-t\xi\right)  f\left(
\xi\right)  \right] 
\\ & \quad
 +\int_{\mathbb{R}_{*}^{d}}x_j\eta\left(  x\right)
\varphi\left(  x\right)  dx \,.
\end{align*}

The final result can be interpreted in the sense of distributions as a vector equation
\[
\operatorname{div}_x\left(  \int_{0}^{1}dt \int_{\left\{ 0<\xi<\frac{x}%
{t}\right\} } d\xi\, \xi\otimes\left(  x-t\xi\right)  K\left(  x-t\xi,\xi\right)
f\left(  x-t\xi\right)  f\left(  \xi\right)  \right)  =x\eta\left(  x\right)
\]
or in a more detailed manner for each of the coordinates 
\begin{equation}
\operatorname{div}\left(  J_{j}\left(  x\right)  \right)  =x_{j}\eta\left(
x\right) \,, \quad j=1,2,\ldots ,d \,,\label{B4}%
\end{equation}%
where each $J_j$ itself is a vector-valued distribution with 
\begin{equation}
(J_{j})_i\left(  x\right)  =\int_{0}^{1}dt \int_{\left\{ 0<\xi<\frac{x}{t}\right\}
} d\xi\,\xi_i\left(  x_{j}-t\xi_{j}\right)  K\left(  x-t\xi,\xi\right)  f\left(
x-t\xi\right)  f\left(  \xi\right)  \ \ ,\ \ j=1,2,\ldots ,d \,.\label{B5}%
\end{equation}
In the case of constant flux solutions, i.e., in the absence of the source term
$\eta,$ equation (\ref{B4}) becomes 
\begin{equation}
\operatorname{div}\left(  J_{j}\left(  x\right)  \right)
=0\,,\quad j=1,2,\ldots ,d \,.\label{B6}%
\end{equation}
Note that the equations (\ref{B4}), (\ref{B6}) indeed correspond to the conservation
laws associated with the transport of each of the components of the clusters of
the system.

In order to quantify the fluxes of different monomer types which characterize
the solutions of (\ref{B6}) we introduce the following notation. We will
write 
\[
\Sigma_{R}=\left\{  x\in\mathbb{R}_{*}^{d}:\left\vert x\right\vert =R\right\}
\text{ for each }R>0\,.
\]
Note that then $\Sigma_{R}= R \Delta^{d-1}$.  The outward-pointing unit 
vector  $n$, with respect to the simplex $\{x\in \R_+^d:0\le |x|\le R\}$, is given 
at any point of $\Sigma_R$ by 
\[
n=\frac{1}{\sqrt{d}}\left(  1,1,\ldots ,1\right) \,.
\]

Let us for simplicity assume that each $J_j(x)$ is a regular function which satisfies (\ref{B6}) 
and is zero if $x_i\le 0$ for any component $i$.  We integrate  (\ref{B6}) 
over the set $\{x\in \R_+^d:R_1\le |x|\le R_2\}$, for arbitrary $0<R_1<R_2$ and use Stokes' theorem.  This shows that there is $A\in \R^d$ such that 
\begin{equation}
\int_{\Sigma_{R}}\left[  J_{j}\left(  x\right)  \cdot n\right]  dS_{x}%
=A_{j}\,, \quad \text{for all }j=1,2,\ldots ,d,\ R>0 \label{B7}%
\end{equation}
where $dS_{x}$ is the surface area element. It readily follows from (\ref{B5})
that $A_{j}\geq0$ for each $j\in\left\{  1,2,\ldots ,d\right\}$, i.e., $A\in \R_+^d$.
In particular, we find that the flux of monomers of
type $j$ is constant across all the surfaces $\Sigma_{R}$.

In contrast to the case $d=1$, finding $f$ for which equations (\ref{B7}) hold does not imply 
that $f$ satisfies (\ref{B6}). This is due to the fact that in the case $d=1$
the set $\Sigma_{R}$ is just a point for each $R>0.$ If $d>1$ the relation
(\ref{B7}) does not specify the distribution of the fluxes $J_{j}\left(
x\right)  $ in each surface $\Sigma_{R}$ and this distribution must be
obtained from the equations (\ref{B6}).

We prove in \cite{FLNV1} that the solutions
of (\ref{B5}), (\ref{B6}) are Dirac-like measures $f$ supported along a line
$\left\{  x=\lambda b:\lambda>0\right\}  $ for some vector 
$b\in\mathbb{R}_{+}^{d}$ with $\left\vert b\right\vert =1$.
Let us point out that indeed there exist
solutions with that form.  To this end it is convenient to reformulate
(\ref{B6}) in weak form and to change to the coordinate system $\left(
\left\vert x\right\vert ,\theta\right)  $ indicated above.

Taking into account (\ref{B8}), it is natural to define a weak solution of
(\ref{B6}) as a measure $f\in\mathcal{M}\left(  \mathbb{R}_{*}^{d}\right)  $
satisfying 
\begin{equation}
0=\int_{\mathbb{R}_{*}^{d}}d\xi\int_{\mathbb{R}_{*}^{d}}dx\left[
\varphi\left(  x+\xi\right)  -\varphi\left(  x\right)  \right]  K\left(
x,\xi\right) x_j f\left(  x\right)  f\left(  \xi\right)  \label{C1}%
\end{equation}
for each $j=1,2,\ldots ,d$ and every
test function $\varphi\in C^{1}_c\left(  \mathbb{R}_{*}^{d}\right)$  
(see Section \ref{StatDef} for a precise definition of weak solutions).

\subsubsection*{Reformulation of the problem \eqref{C1} using a suitable change of variables} %$(r,\theta)-$variables

It is convenient to rewrite (\ref{C1}) using the new coordinates $\left(  r,\theta\right)$
with $r=|x|>0$ and $\theta = \frac{1}{|x|}x\in \Delta^{d-1}$ for $x\in \R^d_*$.
 The inverse map $\R_*\times \Delta^{d-1}\to \R^d_*$ is given by
\begin{equation}
\label{S1E1}x=r\theta\ \ ,\ \ \ r>0,\ \ \theta\in\Delta^{d-1}.
\end{equation}
We compute the Jacobian of the mapping $x\rightarrow\left(  r,\theta\right)
.$ We use the variables $\theta_{1},\theta_{2},\ldots ,\theta_{d-1}$ to
parametrize the simplex and set then 
\[
\theta_{d}=1-\sum_{j=1}^{d-1}\theta_{j}\,.
\]
Thus, the change of variables is $x\rightarrow\left(  r,\theta_{1},\theta
_{2},\ldots ,\theta_{d-1}\right)$. Therefore, with the above implicit definition of $\theta_d$,
\[
\frac{\partial\left(  x_{1},x_{2},\ldots ,x_{d}\right)  }{\partial\left(
r,\theta_{1},\theta_{2},\ldots ,\theta_{d-1}\right)  }=\left\vert
\begin{array}
[c]{cccccc}%
\theta_{1} & \theta_{2} & \theta_{3} & \ldots  & \theta_{d-1} & \theta_{d}\\
r & 0 & 0 & \ldots  & 0 & -r\\
0 & r & 0 & \ldots  & 0 & -r\\
\ldots  & \ldots  & r & \ldots  & \ldots  & \ldots \\
0 & 0 & 0 & \ldots  & 0 & -r\\
0 & 0 & 0 & 0 & r & -r
\end{array}
\right\vert =D_{d}\left(  r;\theta_{1},\theta_{2},\ldots ,\theta_{d}\right)\,.
\]
We can iterate, developing the determinant by columns. Then 
\[
D_{d}\left(  r;\theta_{1},\theta_{2},\ldots ,\theta_{d}\right)  =\left(
-1\right)  ^{d-1}r^{d-1}\theta_{1}-r D_{d-1}\left(  r;\theta_{2},\theta
_{3},\ldots ,\theta_{d}\right)
\]
Iterating, we arrive to 
\[
D_{d}\left(  r;\theta_{1},\theta_{2},\ldots ,\theta_{d}\right)  =\left(
-1\right)  ^{d-1}r^{d-1}\left(  \theta_{1}+\theta_{2}+\ldots +\theta_{d}\right)
=\left(  -r\right)^{d-1}%
\]
Then 
\[
dx=\left\vert \frac{\partial\left(  x_{1},x_{2},\ldots ,x_{d}\right)  }%
{\partial\left(  r,\theta_{1},\theta_{2},\ldots ,\theta_{d-1}\right)  }\right\vert
drd\theta_{1}d\theta_{2}\ldots d\theta_{d-1}=r^{d-1}drd\theta_{1}d\theta
_{2}\ldots d\theta_{d-1}\,.
\]

We can write $d\theta_{1}d\theta_{2}\ldots d\theta_{d-1}$ in terms of the area
element of the simplex. We just use 
\[
dS\left(  \theta\right)  =\sqrt{1+\left(  \nabla_{\theta}h\right)  ^{2}%
}d\theta_{1}d\theta_{2}\ldots d\theta_{d-1}\,.%
\]
with $\theta_{d}=h\left(  \theta_{1},\theta_{2},\ldots ,\theta_{d-1}\right)$ where $h\left(  \theta_{1},\theta_{2},\ldots ,\theta_{d-1}\right)=1-\sum_{j=1}^{d-1}\theta_{j}$. We
will denote the element of area of the simplex as $d\tau\left(  \theta\right)
$. Explicitly, 
\[
d\tau\left(  \theta\right)  =\sqrt{d}\, d\theta_{1}d\theta_{2}\ldots d\theta_{d-1}%
\]
Thus,
\begin{equation}
dx=\frac{r^{d-1}}{\sqrt{d}}drd\tau\left(  \theta\right) \,. \label{Jac}%
\end{equation}

We can now rewrite (\ref{C1}) using the above results.  
Suppose that $x=r\theta$ and $\xi=\rho\sigma.$ We then have $\left\vert
x+\xi\right\vert =r+\rho.$ On the other hand, 
\[
\frac{x+\xi}{\left\vert x+\xi\right\vert }=\frac{r}{r+\rho}\theta+\frac{\rho
}{r+\rho}\sigma\,.
\]
We now rewrite the coagulation kernel in this set of variables as
\begin{equation}
G\left(r,\rho;\theta,\sigma\right)= K\left(r\theta,\rho\sigma\right)\, .  \label{G1}%
\end{equation}%
We also rewrite the measure $f$ in terms of    
the measure $F\in \mathcal{M}_+(\R_*\times \Delta^{d-1})$ which is
defined as
\begin{equation}
\int \psi(r,\theta)\frac{r^{d-1}}{\sqrt{d}} F(r,\theta)drd\tau\left(  \theta\right) 
=\int \psi(|x|,x/|x|) f(x) dx\,, \label{G2}
\end{equation}
for a given test function $\psi\in  C_c(\R_*\times \Delta^{d-1})$.  Notice that  if $f$ is absolutely continuous with a smooth density, both sides of \eqref{G2} are the same as it can be seen using an elementary change of variables. 
%The normalization with the Jacobian is made to guarantee that if $f$ is absolutely continuous, then
%the respective density functions transform as expected.
Then (\ref{C1}) can be equivalently written as 
\begin{align}
&  \int_{0}^{\infty}r^{d}dr\int_{0}^{\infty}\rho^{d-1}d\rho\int_{\Delta^{d-1}%
}d\tau\left(  \theta\right)  \int_{\Delta^{d-1}}d\tau\left(  \sigma\right)
G\left(  r,\rho;\theta,\sigma\right)  \nonumber\\
&  \qquad\quad\times \left[  \psi\left(  r+\rho,\frac{r}{r+\rho}\theta+\frac{\rho
}{r+\rho}\sigma\right)  -\psi\left(  r,\theta\right)  \right]  \theta_j F\left(
r,\theta\right)  F\left(  \rho,\sigma\right)  =0\,,\label{C2}%
\end{align}
for all $j=1,2,\ldots ,d$ and $\psi\in C_c^{1}(\R_*\times \Delta^{d-1})$.  Notice that in any open bounded set of $\R^d_*$ the change of variables \eqref{S1E1} defines a diffeomorphism. 

%For the next step, given a measure 
%$f\in\mathcal{M}_{+}\left(  \mathbb{R}_{*}^{d}\right) $, it will be convenient to 
%begin using the measures $F\in \mathcal{M}_{+}(\R_*\times \Delta^{d-1})$ defined by (\ref{G2}).  
We observe that the change of variables \eqref{G2}  can be understood rigorously 
via the Riesz--Markov--Kakutani theorem applied to the linear functional
\[
\varphi\mapsto  \int_{\R^d_*} \frac{\sqrt{d}}{|x|^{d-1}}
\varphi(|x|,x/|x|) f(x) dx\,, \qquad \varphi \in C_c(\R_*\times \Delta^{d-1})\, 
\]
which allows to define a measure $F$.   
Furthermore, if $f$ satisfies the assumptions in 
Definition \ref{StatInjDef}, we have that $F$ is supported in $[1,\infty)\times \Delta^{d-1}$ and 
\[
 \int_{\R_*\times \Delta^{d-1}} r^{d-1+\gamma+p} F(r,\theta)dr d\tau(\theta)
<\infty \,.
\]

\subsubsection*{A family of weighted Dirac-$\delta$ solutions}

Suppose that $K$ is continuous and homogeneous with homogeneity $\gamma.$ If
the kernel $K$ satisfies (\ref{Ineq2}), (\ref{Ineq3}) with $\gamma+2p<1$, we claim that 
we
then have a family of solutions of (\ref{C2}) given by the following weighted Dirac 
$\delta$-measures
\begin{equation}
F\left(  r,\theta\right)  =\frac{C_{0}}{r^{\frac{\gamma
+1}{2}+d}}\delta\left(  \theta-\theta_{0}\right)  ,\quad C_{0}>0\,, \label{C4}%
\end{equation}
where $\theta_{0}\in \Delta^{d-1}$ is fixed but arbitrary.
To see this, first note that  $\frac{r}{r+\rho}\theta_{0}+\frac{\rho}{r+\rho}\theta
_{0}=\theta_{0}$, and thus then (\ref{C2}) is equivalent to 
\begin{equation}
( \theta_{0})_j
\int_{0}^{\infty}r^{d}dr\int_{0}^{\infty}\rho^{d-1}d\rho\frac{G\left(
r,\rho;\theta_{0},\theta_{0}\right) }{r^{\frac{\gamma+1}{2}+d}%
\rho^{\frac{\gamma+1}{2}+d}}\left[  \psi\left(  r+\rho,\theta_{0}\right)
-\psi\left(  r,\theta_{0}\right)  \right]  =0.\label{eq:C2bis}%
\end{equation}
Notice that the integral in \eqref{eq:C2bis} is well defined for $\psi\in
C_{c}^{1}(\R_*\times \Delta^{d-1})$ and $\gamma+2p<1$.

We now rewrite \eqref{eq:C2bis} in a more convenient form. 
First, since $\theta_0\in \Delta^{d-1}$, there is at least one $j$ such that $(\theta_0)_j>0$.
Thus the factor $(\theta_0)_j$  may be dropped from  \eqref{eq:C2bis}.
Then for any $\psi\in C_{c}^{1}(\R_*\times \Delta^{d-1})$ we may employ inside the integrand 
the identity
\[
\psi\left(  r+\rho,\theta_{0}\right)  -\psi\left(  r,\theta_{0}\right)
=\int_{r}^{r+\rho}\frac{\partial\psi}{\partial t}\left(  t,\theta_{0}\right)
dt
=\int_{a}^{b} \cf{t\ge r} \cf{t< r+\rho}
\frac{\partial\psi}{\partial t}\left(  t,\theta_{0}\right)
dt.
\]
where $0<a<b$ are such that the support of $\psi$ lies in $[a,b]\times \Delta^{d-1}$.
Therefore, applying Fubini's Theorem, \eqref{eq:C2bis} is seen to be equivalent with
\[
\int_{a}^{b}dt\frac{\partial\psi}{\partial t}\left(
t,\theta_{0}\right)  \int_{0}^{t}r^{d}dr\int_{t-r}^{\infty}\rho^{d-1}%
d\rho\frac{G\left(  r,\rho;\theta_{0},\theta_{0}\right)  }{r^{\frac{\gamma
+1}{2}+d}\rho^{\frac{\gamma+1}{2}+d}}=0\,.
\]
Integrating by parts we obtain 
\begin{equation}
\int_{0}^{\infty}dt \, \psi(  t,\theta_{0})  \frac{\partial
}{\partial t}\left(  \int_{0}^{t}r^{d}dr\int_{t-r}^{\infty}\rho^{d-1}%
d\rho\frac{G\left(  r,\rho;\theta_{0},\theta_{0}\right)  }{r^{\frac{\gamma
+1}{2}+d}\rho^{\frac{\gamma+1}{2}+d}}\right)  =0\,.\label{C3}%
\end{equation}

Since this needs to hold for all allowed $\psi$, we find that it is valid if and only if
there is $J_0\ge 0$ such that
\begin{equation}
J_0= \int_{0}^{t}r^{d}dr\int_{t-r}^{\infty}\rho^{d-1}%
d\rho\frac{G\left(  r,\rho;\theta_{0},\theta_{0}\right)  }{r^{\frac{\gamma
+1}{2}+d}\rho^{\frac{\gamma+1}{2}+d}}\,, \quad \text{ for all }t>0 \,.\label{C3bis}%
\end{equation}
This equation is indeed satisfied if $K$ is a homogeneous kernel, since then $G\left(  \lambda r,\lambda
\rho;\theta_{0},\theta_{0}\right)  =\lambda^{\gamma}G\left(  r,\rho;\theta
_{0},\theta_{0}\right)$ for each $\lambda>0$, and thus
\begin{align*}
&  \int_{0}^{t}r^{d}dr\int_{t-r}^{\infty}\rho^{d-1}d\rho\frac{G\left(
r,\rho;\theta_{0},\theta_{0}\right)  }{r^{\frac{\gamma+1}{2}+d}\rho
^{\frac{\gamma+1}{2}+d}}\\
&  =\frac{t^{2d+1}}{t^{\gamma+1+2d}}\int_{0}^{1}r^{d}dr\int_{1-r}^{\infty}%
\rho^{d-1}d\rho\frac{G\left(  tr,t\rho;\theta_{0},\theta_{0}\right)
}{r^{\frac{\gamma+1}{2}+d}\rho^{\frac{\gamma+1}{2}+d}}\\
&  =\int_{0}^{1}r^{d}dr\int_{1-r}^{\infty}\rho^{d-1}d\rho\frac{G\left(
r,\rho;\theta_{0},\theta_{0}\right)  }{r^{\frac{\gamma+1}{2}+d}\rho
^{\frac{\gamma+1}{2}+d}}%
\end{align*}
is constant.

Using the same procedure to simplify the general case
(\ref{C2}) yields an alternative way of prescribing the fluxes
through the surfaces $\left\{  \left\vert x\right\vert =R\right\}$. This
will be made precise in Section \ref{StatDef} (cf. Definitions
\ref{DefConstFlux}).

\subsubsection*{Stationary solutions with a prescribed concentration of monomers}

As a third possibility used in the literature to obtain stationarity
of solutions to coagulation equation, let us briefly mention 
using, instead of sources, boundary conditions to fix the concentration of monomers
to some given value in 
(\ref{B2Stat}).  
The one-component case has already been considered in \cite{FLNV}, but 
the definition becomes more
involved here due to the fact that we have a multicomponent system.

Explicitly, we would then be interested in solutions of 
\begin{equation}
0=\frac{1}{2}\sum_{\beta<\alpha}K_{\alpha-\beta,\beta}n_{\alpha-\beta}%
n_{\beta}-n_{\alpha}\sum_{\beta>0}K_{\alpha,\beta}n_{\beta}\ \ ,\ \ \left\vert
\alpha\right\vert \notin\left\{  0,1\right\}  \label{B3Stat}%
\end{equation}
We will say that $\left\{  n_{\alpha}\right\}  _{\alpha\in\mathbb{N}_{\ast
}^{d}\setminus\left\{  O\right\}  }$ is a \textit{stationary solution of
(\ref{B3Stat}) with a prescribed concentration of monomers}, if it solves
(\ref{B3Stat}) and in addition it satisfies 
\begin{equation}
\sum_{\left\{  \alpha:\alpha_{j}=1\right\}  }n_{\alpha}=c_{j}\ \ \text{for
each }j=1,2,3,\ldots ,d \label{B4Stat}%
\end{equation}
for a given set of concentrations $\left\{  c_{j}\right\}  _{j=1}^{d} \in \R_+^d.$

The existence of solutions of the problem (\ref{B3Stat}), (\ref{B4Stat}) for a
given set of concentrations $\left\{  c_{j}\right\}  _{j=1}^{d}$ is not
evident at all. If we have an injection solution to (\ref{B2Stat}) for a set
of sources $\left\{  s_{\alpha}\right\}  _{\left\{  \left\vert \alpha
\right\vert =1\right\}  },$ then we have a solution of (\ref{B3Stat}),
(\ref{B4Stat}) for the corresponding values of $c_{j}$ obtained by means of
the sum (\ref{B4Stat}). However, there is not any reason to expect that any
set of concentrations $\left\{  c_{j}\right\}  _{j=1}^{d}$ could be obtained
by means of a suitable choice of sources $\left\{  s_{\alpha}\right\}
_{\left\{  \left\vert \alpha\right\vert =1\right\}  }.$ It has been seen in
\cite{FLNV} that in the case $d=1$ the problem (\ref{B3Stat})--(\ref{B4Stat})
can be solved if $\left\vert \gamma+2\lambda\right\vert <1$ and kernels with
the form \eqref{PrevIned1}, \eqref{PrevIneq2}.

\subsection{Rigorous definition of the classes of steady state
solutions\label{StatDef}}

We define now in a precise mathematical way the solutions that we will
consider in this paper.  

\begin{definition}\label{StatInjDef}
Let $\eta\in\mathcal{M}_{+,b}\left(  \mathbb{R}_{*}^{d}\right)  $
with support contained in  $\defset{x\in \R^d_*}{1\le |x|\le L}$ 
for some $L>1.$ Suppose that the coagulation kernel $K$ is continuous and satisfies (\ref{Ineq2}),
(\ref{Ineq3}). We say that $f\in\mathcal{M}_{+}\left(  \mathbb{R}_{*}%
^{d}\right)  $ is a stationary injection solution to (\ref{eq:contStat}) if
the support of $f$ is contained in $\defset{x\in \R^d_*}{|x|\ge 1}$ and $f$ satisfies
\begin{equation}
\int_{\R^d_*} |x|^{\gamma+p} f(dx)<\infty ,\label{S1E4}
\end{equation}
as well as 
\begin{equation}
0=\frac{1}{2}\int_{{{{\mathbb{R}}}_{*}^{d}}}\int_{{{{\mathbb{R}}}_{*}^{d}}%
}K\left(  x,y\right)  \left[  \varphi\left(  x+y\right)  -\varphi\left(
x\right)  -\varphi\left(  y\right)  \right]  f\left(  dx\right)  f\left(
dy\right)  +\int_{{{{\mathbb{R}}}_{*}^{d}}}\varphi\left(  x\right)
\eta\left(  dx\right)  \label{WeakForm1}%
\end{equation}
for any test function 
$\varphi\in C^{1}_c\left(  {{{\mathbb{R}}}_{*}^{d}}\right)$.
\end{definition}

We recall the notation $\R^d_*:=\R^d_{+}\setminus \{O\}$ and that $|x|=\sum_j |x_j|$ 
denotes the $\ell_1$-norm in $\R^d_*$.

Stationary injection solutions for the discrete equation
(\ref{B2Stat}) can be 
considered as solutions $f$ of (\ref{eq:contStat}) with $f$  supported on the
elements of $\mathbb{N}_{*}^{d}=\N_0^d\setminus\{O\}$ by using Dirac-$\delta$ measures as explained next.
Let the sequence
$\left\{  n_{\alpha}\right\}  _{\alpha\in\mathbb{N}_{*}^{d}}$ be a solution of
(\ref{B2Stat}), i.e., it satisfies
\begin{equation}
0=\frac{1}{2}\sum_{\beta<\alpha}K_{\alpha-\beta,\beta}n_{\alpha-\beta}%
n_{\beta}-n_{\alpha}\sum_{\beta>0}K_{\alpha,\beta}n_{\beta}+s_{\alpha
}\,, \qquad \alpha\in\mathbb{N}_{*}^{d} \,, \label{S1E5}%
\end{equation}
where the source satisfies $s_\alpha=0$ whenever $|\alpha|>L$  for some $L>1$.
We then define 
\begin{equation}
f\left(  x\right)  =\sum_{\alpha\in\mathbb{N}_{*}^{d}}n_{\alpha}\delta\left(
x-\alpha\right)  \label{S1E2}%
\end{equation}
as well as 
\begin{equation}
\eta\left(  x\right)  =\sum_{\alpha\in\mathbb{N}_{*}^{d}}s_{\alpha}%
\delta\left(  x-\alpha\right) \,. \label{S1E3}%
\end{equation}
Then $\eta$ satisfies 
the assumptions of Definition
\ref{StatInjDef} with the same parameter $L$ and we can  define a solution of (\ref{S1E5}) as follows.
\begin{definition}
\label{DefStatInjDis}Suppose that $\left\{  s_{\alpha}\right\}  _{\alpha
\in\mathbb{N}_{*}^{d}  }$ is a non-negative sequence supported
in a finite collection of values $\alpha.$ A sequence $\left\{
n_{\alpha}\right\}  _{\alpha\in\mathbb{N}_{*}^{d}
}$, with $n_{\alpha}\geq0$ for $\alpha\in\mathbb{N}_{*}^{d}$, is a stationary injection solution of (\ref{S1E5}) if
\[
 \sum_{\alpha\in\mathbb{N}_{*}^{d}} |\alpha|^{\gamma+p}n_{\alpha} < \infty\,, 
\]
and
the measure $f\in\mathcal{M}_{+}\left(  \mathbb{R}_{*}^{d}\right)  $ defined
as in (\ref{S1E2}) solves (\ref{eq:contStat}) in the sense of Definition
\ref{StatInjDef}.
\end{definition}
Thanks to the assumptions on $n_\alpha$ the measure
$f$ has its support in $\defset{x\in \R^d_*}{|x|\ge 1}$ and satisfies (\ref{S1E4}). 

\medskip

We now provide a rigorous definition of the constant flux to
the equation (\ref{eq:contStat}) with $\eta=0$, namely 
\begin{equation}
0=\frac{1}{2}\int_{\left\{ 0<y<x\right\} }K\left(  x-y,y\right)  f\left(
x-y,t\right)  f\left(  y,t\right)  dy-\int_{\mathbb{R}_{*}^{d} }K\left(
x,y\right)  f\left(  x,t\right)  f\left(  y,t\right)  dy\,. \label{S1E7}%
\end{equation}

\begin{definition}\label{DefConstFlux}
Suppose that the coagulation kernel $K$ is continuous and satisfies
(\ref{Ineq2}), (\ref{Ineq3}). 
We say that $f\in\mathcal{M}\left(
\mathbb{R}_{*}^{d}\right)  $ is a stationary solution to (\ref{S1E7}) 
if 
\begin{equation*}
\int_{\defset{x\in \R^d_*}{|x|\ge 1}} |x|^{\gamma+p} f(dx)
+
\int_{\defset{x\in \R^d_*}{|x| < 1}} |x|^{1-p} f(dx)
<\infty\label{cond_constFlux}
\end{equation*}
is satisfied and 
(\ref{WeakForm1}) holds with $\eta=0$,  
for every test function 
$\varphi\in C^{1}_c\left(  {{{\mathbb{R}}}_{*}^{d}}\right)$.

We define the \defem{total flux across
the surface $\left\{  \left\vert x\right\vert =R\right\}  $}
as the vector-valued function $A(R)\in \R_+^d$, $R>0$, defined by means of 
\begin{equation}
A_j(R) =  \frac{1}{d}
\int_{(0,R]\times \Delta^{d-1}}dr d\tau\left(  \theta\right)\int_{(R-r,\infty)\times \Delta^{d-1}} d\rho d\tau\left(  \sigma\right)
r^{d}\rho^{d-1}
F\left(  r,\theta\right)  F\left(  \rho,\sigma\right)  \theta_j G\left(
r,\rho;\theta,\sigma\right) \,,\label{S1E8}
\end{equation}
where the function $G$ is as in (\ref{G1}) 
and the measure $F$ has been defined using (\ref{G2}). 
We say that $f$ is a \defem{non-trivial constant flux solution} of (\ref{S1E7}) if it is a stationary solution and there is  
$J_0>O$ 
such that $A(R)=J_0$ for all  $R>0$.
\end{definition}

\begin{remark}
The above definition of flux, (\ref{S1E8}), is obtained by a similar computation as 
leading to the special case in (\ref{C3bis}) with the additional assumption that the test-function is constant in the simplex-variable $\theta$; the details of this argument may be found in the proof of Theorem \ref{ThmExistCont}. 
Note that in the one-component case to impose that the fluxes are constant,
i.e. \eqref{S1E8}, implies that $f$ is a solution to the coagulation equation.
This does not automatically happen in the multicomponent case and this explains why we need
to further assume \eqref{WeakForm1} here.
\end{remark}

\section{Main results} 
\label{sec:4}

We state in this Section the 
main results proven in this paper. They can be thought as a  natural extensions from one-~to multi-component systems of 
the existence and non-existence of stationary
injection solutions to (\ref{eq:contStat}) and (\ref{S1E5}),
as well as of the constant
flux solutions to (\ref{S1E7}), which have been considered in
\cite{FLNV} for the one-component case. 

We first describe the existence results for
the injection solutions:
\begin{theorem}
\label{ThmExistCont}
Suppose that the coagulation kernel $K$ is a continuous symmetric function that satisfies
(\ref{Ineq2}), (\ref{Ineq3}) with $\gamma+2p<1.$ Suppose that 
$\eta\in\mathcal{M}_{+,b}\left(  \mathbb{R}_{*}^{d}\right)$ has 
its support in the set $\defset{x\in \R^d_*}{1\le |x|\le L}$ for some $L>1$. 
Then, there exists a stationary injection solution $f\in\mathcal{M}%
_{+}\left(  \mathbb{R}_{*}^{d}\right)  $ to (\ref{eq:contStat}) in the sense
of Definition \ref{StatInjDef}.
\end{theorem}

The following theorem is a corollary of the previous result assuming that $f$ is supported on the set  $\mathbb{N}_{*}^{d}$. However, since the discrete coagulation equation has an independent interest and it is relevant for applications, we formulate the result as a separate theorem. 

\begin{theorem}
\label{ThmExistDisc}
Suppose that the coagulation kernel $K$ is symmetric and satisfies (\ref{Ineq1}), (\ref{Ineq3})
with $\gamma+2p<1.$ Let $\left\{  s_{\alpha}\right\}  _{\alpha\in
\mathbb{N}_{*}^{d}  }$ be a non-negative sequence
supported on a finite number of values $\alpha.$ Then, there exists a
stationary injection solution $\left\{  n_{\alpha}\right\}  _{\alpha
\in\mathbb{N}_{*}^{d}  }$ to (\ref{S1E5}) in the
sense of Definition \ref{DefStatInjDis}.
\end{theorem}

\begin{remark}
Notice that the assumptions on the kernels $K$ in Theorems \ref{ThmExistCont},
\ref{ThmExistDisc} are more general than those in \cite{FLNV}, since the
assumptions on the kernels (\ref{PrevIned1}), (\ref{PrevIneq2}) were used
there. Therefore, the results in this paper provide an improvement of the
earlier results even in the case of one component, $d=1$.
\end{remark}

\begin{remark}
Let us point out that no uniqueness of the solutions
is claimed in Theorems \ref{ThmExistCont}, \ref{ThmExistDisc}. The issue of uniqueness of stationary injection
solutions is an interesting open problem.
\end{remark}

The restrictions for the values of $\gamma$ an $p$ in Theorems \ref{ThmExistCont}, \ref{ThmExistDisc} are not only
sufficient to have stationary injection solutions, but they are also necessary. Indeed,
we have the following non-existence results, which are analogous to Theorems 2.4, 5.3. in
\cite{FLNV}.

\begin{theorem} 
The following results hold:
\begin{itemize}
\item[(i)] \label{NonexistCont} 
Suppose that $K$ is a continuous symmetric function that satisfies
(\ref{Ineq2}), (\ref{Ineq3}) with $\gamma+2p\geq1.$ 
 Suppose that 
$\eta\in\mathcal{M}_{+,b}\left(  \mathbb{R}_{*}^{d}\right)  $
is supported inside the set $\defset{x\in \R^d_*}{1\le |x|\le L}$ for some $L>1$. 
Then there are no solutions to (\ref{eq:contStat}) in the
sense of Definition \ref{StatInjDef}.

\item[(ii)] \label{NonExistDisc}
Suppose that $K$ is symmetric and satisfies (\ref{Ineq1}), (\ref{Ineq3})
with $\gamma+2p\geq1.$ Let $\left\{  s_{\alpha}\right\}  _{\alpha\in
\mathbb{N}_{*}^{d}  }$ be a nonnegative sequence
supported on a finite number of values $\alpha$.
Then there are no solutions to
(\ref{S1E5}) in the sense of Definition \ref{DefStatInjDis}.
\end{itemize} 
\end{theorem}

\bigskip

Concerning the constant flux solutions to (\ref{S1E7}) we have already seen in Section
\ref{StatTypes} that (\ref{C4}) defines a constant flux solution to
(\ref{S1E7}) in the sense of Definition \ref{DefConstFlux} if $\gamma+2p<1$, at least
when $K$ is a homogeneous kernel function.
If
$\gamma+2p\geq1,$ such solutions do not exist. This is the content of the
following Theorem.
\begin{theorem}
\label{NonexFluxSol}Suppose that $K$ satisfies (\ref{Ineq2}), (\ref{Ineq3})
with $\gamma+2p\geq1$. There are no non-trivial constant flux solutions to (\ref{S1E7}),
in the sense of Definition \ref{DefConstFlux}.
\end{theorem}

\bigskip

\begin{remark}
We observe that we did not require the kernel $K$ to be homogeneous in Theorems \ref{ThmExistCont}, \ref{ThmExistDisc}, nor in Theorems \ref{NonexistCont}, \ref{NonexFluxSol}, for which the upper and lower estimates \eqref{Ineq1}, \eqref{Ineq2}, \eqref{Ineq3} suffice. 
Notice that, if the kernel $K$ is homogeneous, there are constant flux solutions to \eqref{S1E7} with $H\left(  r\right)=r^{-\frac{\gamma+3}{2}}$ (cf.~\eqref{C4}). However,
not all constant flux solutions to the one-component equation are necessarily power-laws, even if one assumes homogeneity of the kernel. Examples of non-power law solutions for certain one-dimensional coagulation kernels are given in \cite{FLNV3}.
\end{remark}

\bigskip

\section{Some auxiliary results} 
\label{sec:5}

\subsection{A convenient reformulation of the problem}  
\label{ReductpZero} 
\bigskip

The kernels $K$ satisfying (\ref{Ineq1})--(\ref{Ineq3}) are 
characterized by the two parameters $\gamma, p$.
It turns out that, using a suitable change of variable, we can reformulate the
problems described in Section \ref{StatTypes} with kernels $K$ 
into similar problems with new kernels $\tilde{K}$ characterized by parameters
$\tilde{\gamma}=\gamma+2p$ and $\tilde{p}=0$.

To see this, we will use an idea used in \cite{DP} (see also \cite{BLL} and the recent paper \cite{Laurencotpreprint20}). 
Before formulating the
precise results, we explain the idea in the case of the continuous coagulation
equation (\ref{eq:contStat}). Suppose that $f\in\mathcal{M}_{+}\left(
\mathbb{R}_{*}^{d}\right)  $ solves (\ref{eq:contStat}). We can rewrite this
equation as 
\begin{align*}
& 0=\frac{1}{2}\int_{\left\{ 0<\xi<x\right\} }d\xi\left[  K\left(  x-\xi
,\xi\right)  \left\vert x-\xi\right\vert ^{p}\left\vert \xi\right\vert
^{p}\right]  \frac{f\left(  x-\xi\right)  }{\left\vert x-\xi\right\vert ^{p}%
}\frac{f\left(  \xi\right)  }{\left\vert \xi\right\vert ^{p}} 
\\ \quad &
- \int
_{\mathbb{R}_{*}^{d}}d\xi\left[  K\left(  x,\xi\right)  \left\vert
x\right\vert ^{p}\left\vert \xi\right\vert ^{p}\right]  \frac{f\left(
x\right)  }{\left\vert x\right\vert ^{p}}\frac{f\left(  \xi\right)
}{\left\vert \xi\right\vert ^{p}}+\eta\left(  x\right) . 
\end{align*}
Therefore, if we multiply the measure $f$ by the strictly positive continuous
function $x\mapsto \left\vert x\right\vert ^{-p}$, i.e., if we 
define 
$h\left(  x\right)dx =
\left\vert x\right\vert ^{-p}f(x)dx$,
we find that it solves the following equation 
\begin{equation}
0=\frac{1}{2}\int_{\left\{ 0<\xi<x\right\} } d\xi\tilde{K}\left(  x-\xi
,\xi\right)  h\left(  x-\xi\right)  h\left(  \xi\right)  -\int_{{\mathbb{R}%
}_{+}^{d}} d\xi\tilde{K}\left(  x,\xi\right)  h\left(  x\right)  h\left(
\xi\right)  +\eta\left(  x\right)  \label{S1E9}%
\end{equation}
where 
\begin{equation}
\tilde{K}\left(  x,y\right)  =K\left(  x,y\right)  \left\vert x\right\vert
^{p}\left\vert y\right\vert ^{p} .\label{S2E1}%
\end{equation}

Notice that (\ref{S1E9}) has the same form as (\ref{eq:contStat}). However,
the bounds for the kernel $\tilde{K}$ are simpler than those for $K$. 
Namely, we recall
that $K$ satisfies (\ref{Ineq2}), (\ref{Ineq3}), and thus obtain the bounds
\[
c_{1}\left(  \left\vert x\right\vert +\left\vert y\right\vert \right)
^{\gamma}\left\vert x\right\vert ^{p}\left\vert y\right\vert ^{p}\Phi\left(
\frac{\left\vert x\right\vert }{\left\vert x\right\vert +\left\vert
y\right\vert }\right)  \leq\tilde{K}\left(  x,y\right)  \leq c_{2}\left(
\left\vert x\right\vert +\left\vert y\right\vert \right)  ^{\gamma}\left\vert
x\right\vert ^{p}\left\vert y\right\vert ^{p}\Phi\left(  \frac{\left\vert
x\right\vert }{\left\vert x\right\vert +\left\vert y\right\vert }\right)\,.
\]
Denoting $s=\frac{\left\vert x\right\vert }{\left\vert x\right\vert
+\left\vert y\right\vert }$, we then obtain 
\[
c_{1}\left(  \left\vert x\right\vert +\left\vert y\right\vert \right)
^{\gamma+2p}s^{p}\left(  1-s\right)  ^{p}\Phi\left(  s\right)  \leq\tilde
{K}\left(  x,y\right)  \leq c_{2}\left(  \left\vert x\right\vert +\left\vert
y\right\vert \right)  ^{\gamma+2p}s^{p}\left(  1-s\right)  ^{p}\Phi\left(
s\right)\,.
\]
By (\ref{Ineq3}), this implies 
\begin{equation}
c_{1}\left(  \left\vert x\right\vert +\left\vert y\right\vert \right)
^{\gamma+2p}\leq\tilde{K}\left(  x,y\right)  \leq c_{2}\left(  \left\vert
x\right\vert +\left\vert y\right\vert \right)  ^{\gamma+2p} \,,  \qquad x,y\in \R^d_*\,.\label{S2E2}%
\end{equation}

Therefore, $h$ solves (\ref{S1E9}) which is the same equation as
(\ref{eq:contStat}) with a kernel $\tilde{K}$ satisfying (\ref{S2E2}).
The new kernel thus satisfies (\ref{Ineq2}) after
replacing $\gamma$ by $\tilde{\gamma
}=\gamma+2p$ and $p$ by $\tilde{p}=0$. 
In addition, the supports of $h$ and $f$ are the same, and the moment bound 
(\ref{S1E4}) is true for $f$, $\gamma$, and $p$, if and only if it is true for
$h$, $\tilde{\gamma}$, and $\tilde{p}$. 
Notice that exactly the same argument can be made in the weak
formulation of (\ref{eq:contStat}), as well as in the discrete problem
(\ref{B2Stat}). Therefore we can reduce the discrete problem considered in
this paper to an analogous problem with kernel
\begin{equation}
\label{S2E1bis}\tilde{K}\left(  \alpha,\beta\right)  =K\left(  \alpha
,\beta\right)  \left\vert \alpha\right\vert ^{p}\left\vert \beta\right\vert
^{p}%
\end{equation}
which satisfies an estimate 
\begin{equation}
c_{1}\left(  \left\vert \alpha\right\vert +\left\vert \beta\right\vert
\right)  ^{\gamma+2p}\leq\tilde{K}\left(  \alpha,\beta\right)  \leq
c_{2}\left(  \left\vert \alpha\right\vert +\left\vert \beta\right\vert
\right)  ^{\gamma+2p}.\label{S2E3}%
\end{equation}

These observations can be summarized as follows.
\begin{lemma}
\label{EquivDefin}

The following statements hold:

\begin{itemize}
\item[(i)] Let 
$\eta\in\mathcal{M}_{+,b}\left(  \mathbb{R}_{*}^{d}\right)  $ with the 
support in the set $\defset{x\in \R^d_*}{1\le |x|\le L}$ for some $L>1$. 
Let us assume that $K$ is continuous and satisfies (\ref{Ineq2}), (\ref{Ineq3}).
Then, the Radon measure $f\in\mathcal{M}_{+}\left(  \mathbb{R}_{*}^{d}\right)  $ is
a stationary injection solution to (\ref{eq:contStat}) in the sense of
Definition \ref{StatInjDef} if and only if the Radon measure $h\left(
x\right)  =\frac{f\left(  x\right)  }{\left\vert x\right\vert ^{p}}$ is a
stationary injection solution to (\ref{eq:contStat}) with  kernel
$\tilde{K}$ defined as in (\ref{S2E1}). Moreover, (\ref{S2E2}) holds.

\item[(ii)] Suppose that $\left\{  s_{\alpha}\right\}  _{\alpha\in
\mathbb{N}_{\ast}^{d}\setminus\left\{  O\right\}  }$ is a sequence supported
in a finite number of values $\alpha.$ Let us assume that the kernel $K$
satisfies (\ref{Ineq1}), (\ref{Ineq3}). Then the sequence $\left\{  n_{\alpha
}\right\}  _{\alpha\in\mathbb{N}_{*}^{d}  }$ is a
stationary injection solution to (\ref{S1E5}) in the sense of Definition
\ref{DefStatInjDis} if and only if the sequence $\left\{ \left\vert \alpha\right\vert ^{-p}
n_{\alpha}\right\}  _{\alpha\in\mathbb{N}_{*}^{d} }$ is a stationary injection solution to
(\ref{S1E5}) with kernel $\tilde{K}$ defined as in (\ref{S2E1bis}). Moreover, \eqref{S2E3} is satisfied.

\item[(iii)] Let us assume that $K$ is continuous and satisfies (\ref{Ineq2}),
(\ref{Ineq3}). The Radon measure $f\in\mathcal{M}_{+}\left(  \mathbb{R}%
_{+}^{d}\right)  $ is a constant flux solution to (\ref{S1E7}) in the sense of
Definition \ref{DefConstFlux} if and only if the Radon measure $h\left(
x\right)  =\frac{f\left(  x\right)  }{\left\vert x\right\vert ^{p}}$ is a
constant flux solution to (\ref{S1E7}) with the kernel $\tilde{K}$ defined as in
(\ref{S2E1}). Moreover, (\ref{S2E2}) is satisfied.
\end{itemize}
\end{lemma}

Therefore, it is sufficient to consider kernels satisfying (\ref{Ineq1})--(\ref{Ineq3}) 
with $p=0$. Equivalently, we need to examine kernels
satisfying (\ref{S2E2}), (\ref{S2E3}).

\subsection{An auxiliary Lemma}

The following result which will be extensively used in the rest of the paper
has been proven as Lemma 2.10 in \cite{FLNV}. It allows to transform
estimates of averaged integrals into estimates on the whole line.
\begin{lemma}\label{lem:bound}
Suppose $a>0$ and $b \in (0,1)$, and assume that 
$R\in (0,\infty]$ is such that 
$R\ge a$.  Consider some $f \in \mathcal{M_+}(\R_*)$ and $\varphi \in C(\R_*)$,
with $\varphi\ge 0$.
\begin{enumerate}
 \item\label{it:Rlessinf} Suppose $R<\infty$, and assume that there is 
 $g \in L^1([a,R])$ such that $g\ge 0$ and 
\begin{equation}\label{eq:avg}
\frac{1}{z}\int_{[bz,z]} \varphi(x) f(dx) \leq g(z)\,, \quad \text{for } z \in [a,R] \,.
\end{equation}
Then
\begin{equation}\label{eq:bound}
\int_{[a,R]} \varphi(x) f(dx) \leq \frac{\int_{[a,R]}g(z)dz}{\vert \ln b\,\vert }
+ R g(R)\,.
\end{equation}
 \item\label{it:polcase}
 Consider some $r\in (0,1)$, and assume that $a/r\le R<\infty$.
 Suppose that 
(\ref{eq:avg}) holds for $g(z)=c_0 z^q$, with $q\in \R$ and $c_0\ge 0$.  Then there is a constant $C>0$, which depends only on $r$, $b$ and $q$, such that
\begin{equation}\label{eq:polbound}
\int_{[a,R]} \varphi(x) f(dx) \leq C c_0 \int_{[a,R]} z^q dz \,.
\end{equation}
\item\label{it:Risinf} If $R=\infty$ and there is 
 $g \in L^1([a,\infty))$ such that $g\ge 0$ and 
\begin{equation}\label{eq:avgRinfty}
\frac{1}{z}\int_{[bz,z]} \varphi(x) f(dx) \leq g(z)\,, \quad \text{for } z \ge a \,,
\end{equation}
then
\begin{equation}\label{eq:boundRinfty}
\int_{[a,\infty)} \varphi(x) f(dx) \leq \frac{\int_{[a,\infty)}g(z)dz}{\vert \ln b\,\vert }\,.
\end{equation}
\end{enumerate}
\end{lemma}

\bigskip

\section{Proof of the existence of stationary injection solutions} 
\label{sec:6}

In this Section we prove Theorems \ref{ThmExistCont}, \ref{ThmExistDisc}.

\subsection{Continuous coagulation equation\label{ExistContCase}}

We first prove Theorem \ref{ThmExistCont}. Notice that due to Lemma
\ref{EquivDefin} it is enough to prove the Theorem under the additional assumptions 
$\gamma<1$, $p=0$.  In particular, then the coagulation kernel satisfies (\ref{S2E2}) with $p=0$.

We will follow a strategy that has been  used in the literature  to show existence results for some classes of unbounded coagulation kernels (cf.\ \cite{BLL}). This consists in proving first the existence of stationary
injection solutions for a truncated version of the problem in which the kernel
$K$ is replaced by a compactly supported kernel. We will then derive estimates
for the solutions of these truncated problems that are uniform in the
truncation parameter and we can then take the limit in the truncated problem
and derive a solution to (\ref{eq:contStat}). 

\bigskip

We first define the truncated kernel. We will make two truncations, the first
one to obtain a bounded kernel and the second one to obtain a kernel with
compact support. Before describing these truncations in detail we prove that
there exists a stationary injection solution for a large class of coagulation
equations with bounded, compactly supported kernels. This result will be used
later as an auxiliary tool.

We will use the following Assumptions to characterize a class of solutions in a simplified 
setup where $M$ is a cutoff parameter, $K=K_M$ is a suitably bounded kernel, and 
we additionally cut off the ``gain term'' for large values of $|x|$.  This will result in unique solvability of the coagulation evolution equation, and imply existence of
stationary solutions.

\begin{assumption}[cutoff model]
\label{Assumptions} We will make the following assumptions on the fixed source
term $\eta$, on the kernel $K_{M}$ and on the cutoff function $\zeta_{M}$.

\begin{itemize}
\item[(i)] Consider a source term $\eta\in\mathcal{M}_{+,b}\left( {{\mathbb{R}}%
}_{*}^{d}\right) $. There exists a real number $L\geq1$ such that
$\operatorname*{supp}\left(  \eta\right)  \subset\left\{  x\ : \ 1\leq
\left\vert x\right\vert \leq L\right\} $.

\item[(ii)] The kernel $K_{M}:\mathbb{R}_{*}^{d}\times\mathbb{R}_{*}%
^{d}\rightarrow{\mathbb{R}}_{+}$ is a continuous, nonnegative, symmetric
function. Suppose that $M$, $a_{1}$, $a_{2}$ are constants such that $M>2 L$, with $L$
as in item $(i)$, and $0<a_{1}<a_{2}$.  Assume that the kernel $K_{M}$ satisfies
\[
K_{M}(x,y)\leq a_{2}\quad\text{for all}\quad(x,y) \in\big(\mathbb{R}_{*}%
^{d}\big)^{2}.
\]
We assume also
\[
K_{M}(x,y)\in\left[ a_{1},a_{2}\right]  \quad\text{for} \quad(\vert
x\vert,\vert y\vert)\in\left[ 1,M\right] ^{2}
\]
and
\[
K_{M}(x,y)=0 \quad\text{if}\quad(\left\vert x\right\vert , \left\vert
y\right\vert ) \notin\big(0,2 M\big)^{2}.
\]
 
\item[(iii)] We assume that $\zeta_{M}$ is a fixed cutoff function 
such that $\zeta_{M}\in C\left(  \mathbb{R}_{*}^{d}\right)  $,
\[
0\leq\zeta_{M}\leq1, \;\; \zeta_{M}\left(  x\right)  =1\;\; \text{for} \;\;
0<\left\vert x\right\vert \leq \frac{1}{2}M ,\;\;\; \text{and }\;\; \zeta_{M}\left(
x\right)  =0\;\;\text{for}\;\;\left\vert x\right\vert \geq M.
\]

\end{itemize}
\end{assumption}

The cutoff function $\zeta_M$ 
will be used to inactivate the ``gain term'' for large cluster sizes.
Explicitly, in the simplified problem we study solutions to the 
evolution equations 
\begin{align}
& \partial_{t}f(x,t)=\frac{\zeta_{M}\left(  x\right)  }{2}\int_{\left\{  0<y<
x\right\}  }K_{M}(x-y,y)f(x-y,t)f(y,t)dy
\nonumber \\ & \qquad
-\int_{{{{\mathbb{R}}}_{*}^{d}}}%
K_{M}(x,y)f(x,t)f(y,t)dy+\eta(x)\,, \label{evolEqTrunc}% 
\end{align}
where the kernel $K_{M}$, the source $\eta$, and the function $\zeta_{M}$ are
as in Assumption \ref{Assumptions}. In particular we are interested in the
steady states associated to (\ref{evolEqTrunc}). These satisfy 
\begin{equation}
0=\frac{\zeta_{M}\left(  x\right)  }{2}\int_{\left\{  0<y< x\right\}  }%
K_{M}(x-y,y)f(x-y,t)f(y,t)dy-\int_{{{{\mathbb{R}}}_{*}^{d}}}K_{M}%
(x,y)f(x,t)f(y,t)dy+\eta(x)\,. \label{SteadTrunc}%
\end{equation}

We will restrict our attention to solutions of (\ref{SteadTrunc}) which vanish
for small and very large cluster sizes. 
More precisely, we will use the
following concept of solution to (\ref{SteadTrunc}).
\begin{definition}
\label{DefStatSol} Suppose that Assumption \ref{Assumptions} holds. We will
say that $f\in{\mathcal{M}_{+}(\mathbb{R}_{*}^{d})}$, satisfying $f\left(
\left\{  x\in{\mathbb{R}_{*}^{d}}:|x|<1\text{ or }\left\vert x\right\vert
>M\right\}  \right)  =0$, is a stationary injection solution to
(\ref{evolEqTrunc}) if the following identity holds for any test function
$\varphi\in C_{c}\left(  \mathbb{R}_{*}^{d}\right)  $:
\begin{align}
0 = &  \ \frac{1}{2}\int_{{{{\mathbb{R}}}_{*}^{d}}}\int_{{{{\mathbb{R}}}%
_{*}^{d}}}K_{M}\left(  x,y\right)  \left[  \varphi\left(  x+y\right)
\zeta_{M}\left(  x+y\right)  -\varphi\left(  x\right)  -\varphi\left(
y\right)  \right]  f\left(  dx\right)  f\left(  dy\right) 
+\int_{{{{\mathbb{R}}}_{*}^{d}}}\varphi\left(  x\right)  \eta\left(
dx\right)  . \label{eq:evol_eqWeakSt}%
\end{align}

\end{definition}

In order to prove the existence of a stationary injection solution to
(\ref{evolEqTrunc}) in the sense of Definition \ref{DefStatSol} we will obtain these solutions as a fixed point for the
corresponding evolution problem.  
We first prove the following result.

\begin{proposition}
\label{thm:existence_evolution} Suppose that Assumption \ref{Assumptions}
holds. Then, for any initial condition $f_{0}$ satisfying $f_0\left(
\left\{  x\in{\mathbb{R}_{*}^{d}}:|x|<1\text{ or }\left\vert x\right\vert
>M\right\}  \right)  =0$ and for
any $T>0,$ there exists a unique time-dependent solution $f\in{C^{1}(\left[
0,T\right]  ,\mathcal{M}_{+,b}(\mathbb{R}_{*}^{d}))}$ to (\ref{evolEqTrunc})
which solves it in the classical sense. Moreover, we have%
\begin{equation}
f\left(  \left\{  x\in{\mathbb{R}_{*}^{d}}:|x|<1\text{ or }\left\vert x\right\vert
>M\right\}   ,t\right)  =0\,,\quad\text{for }0\leq t\leq T\,,
\label{fVanish}%
\end{equation}
and the following estimate holds
\begin{equation}
\int_{\mathbb{R}_{*}^{d}}f(dx,t)\leq\int_{\mathbb{R}_{*}^{d}}f_{0}%
(dx)+Ct\,,\quad t\geq0\,, \label{fEstim}%
\end{equation}
for $C=\int_{\mathbb{R}^d_{*}}\eta(dx)\geq0$ which is independent of $f_{0}$,
$t$, and $T$.
\end{proposition}

The proof of Proposition \ref{thm:existence_evolution} reduces to the reformulation of
(\ref{evolEqTrunc}) as an integral equation by means of an application of the
Duhamel principle. The well-posedness result then follows by means of a
standard fixed point argument. The estimate (\ref{fEstim}) is a consequence of
the inequality $\partial_{t}\left(  \int_{\mathbb{R}_{*}^{d}}f(dx,t)\right)
\leq\int_{\mathbb{R}^d_{*}}\eta(dx)$ which follows by integration of
(\ref{evolEqTrunc}). Given that the argument is just a small adaptation of the
similar one in \cite[Proposition 3.6]{FLNV} we will skip the details of the proof.

The solution $f$ obtained in Proposition
\ref{thm:existence_evolution} has a number of useful properties needed later: 
\begin{proposition}
\label{PropSolutTimeDep}Suppose that Assumption \ref{Assumptions} holds. Let
$f_{0}\in{\mathcal{M}_{+,b}(\mathbb{R}_{*}^{d})}$ be as in the statement of
Proposition \ref{thm:existence_evolution} and let $f\in{C^{1}(\left[  0,T\right]
,\mathcal{M}_{+,b}(\mathbb{R}_{*}^{d}))}$ be the solution to (\ref{evolEqTrunc})
which has been obtained in Proposition \ref{thm:existence_evolution}. Then,
the following properties hold.

\begin{itemize}
\item[(i)] The function $f$ is a weak solution to (\ref{evolEqTrunc}), i.e.,
for a test function $\varphi\in C^{1}(\left[  0,T\right]  ,C_{c}\left(
{\mathbb{R}}_{\ast}^{d}\right)  )$ and any $T>0$ the following identity holds 
\begin{align}
&  \frac{d}{dt}\int_{{\mathbb{R}_{*}^{d}}}\varphi\left(  x,t\right)  f\left(
dx,t\right)  -\int_{{\mathbb{R}_{*}^{d}}}\partial_{t}\varphi\left(
x,t\right)  f\left(  dx,t\right) \nonumber\\
&  =\frac{1}{2}\int_{{\mathbb{R}_{*}^{d}}}\int_{{\mathbb{R}_{*}^{d}}}%
K_{M}\left(  x,y\right)  \left[  \varphi\left(  x+y,t\right)  \zeta_{M}\left(
x+y\right)  -\varphi\left(  x,t\right)  -\varphi\left(  y,t\right)  \right]
f\left(  dx,t\right)  f\left(  dy,t\right) \nonumber\\
&  \quad+\int_{{\mathbb{R}_{*}^{d}}}\varphi\left(  x,t\right)  \eta\left(
dx\right)\,, \qquad \text{for every } t\in\left[  0,T\right] \,. \label{eq:evol_eqWeak}%
\end{align}

\item[(ii)] The following inequality is satisfied%
\begin{equation}
\partial_{t}\left(  \int_{\mathbb{R}_{*}^{d}}f\left(  dx,t\right)  \right)
\leq-\frac{a_{1}}{2}\left(  \int_{\mathbb{R}_{*}^{d}}f\left(  dx,t\right)
\right)  ^{2}+\int_{\mathbb{R}_{*}^{d}}\eta\left(  dx\right)  .
\label{IneqQuad}%
\end{equation}

\item[(iii)] For each $M>1$ we define a topological vector space
\[
\mathcal{X}_{M}=\left\{  f\in{\mathcal{M}_{+,b}(\mathbb{R}_{*}^{d})} : f\left(
\left\{  x\in{\mathbb{R}_{*}^{d}}:|x|<1\text{ or }\left\vert x\right\vert
>M\right\}  \right)  =0\right\}
\]
endowed with the weak topology of measures, i.e., the $*$-weak topology inherited as a closed subspace of $C_0(\R^d_*)^*$.  If $R>0$, the subset $\mathcal{U}_R := \defset{f\in \mathcal{X}_M}{\int_{\R^d_*}f(dx)\le R}$ is compact and metrizable in this topology.
For each $t>0$ we define an operator $S\left(  t\right)
:\mathcal{X}_{M}\rightarrow\mathcal{X}_{M}$ by means of $S\left(  t\right)
f_{0}=f\left(  \cdot,t\right)  ,$ with $f\left(  \cdot,t\right)  $ as in
Proposition \ref{thm:existence_evolution}. Then the family of operators
$\left\{  S\left(  t\right)  \right\}  _{t\geq0}$ define a continuous
semigroup in $\mathcal{X}_{M}.$ More precisely, we have 
\begin{equation}
S\left(  0\right)  =I\ \ ,\ \ S\left(  t_{1}+t_{2}\right)  =S\left(
t_{1}\right)  S\left(  t_{2}\right)  \text{ for each }t_{1},t_{2}\in
\mathbb{R}_{+}\ ,\ \label{SemProps}%
\end{equation}
the mapping $f\mapsto S\left(  t\right)  f$ with $f\in\mathcal{X}_{M}$
is continuous in $\mathcal{X}_{M}$ for each $t\geq0.$ Moreover, the mapping
$t\mapsto S\left(  t\right)  f$ from $\mathbb{R}_{+}$ to $\mathcal{X}_{M}$
is continuous for each $f\in\mathcal{X}_{M}.$
\end{itemize}
\end{proposition}

\begin{proof}
All the results contained in Proposition \ref{PropSolutTimeDep} are small
adaptations of similar results contained in \cite{FLNV}. Therefore, we just
sketch the main ideas and refer to that paper for details.

The proof of $(i)$ is a consequence of the fact that $f$ is a classical
solution to (\ref{evolEqTrunc}). We then compute $\frac{d}{dt}\left(
\int_{{\mathbb{R}^d_{*}}}\varphi\left(  x,t\right)  f\left(  dx,t\right)
\right)  $ and write it in terms of the coagulation kernel using
\eqref{evolEqTrunc}. Using the regularity properties of $\varphi$ and $f$, and
applying Fubini's theorem to rewrite the term containing convolution in
\eqref{evolEqTrunc}, we obtain \eqref{eq:evol_eqWeak}.

Inequality (\ref{IneqQuad}) in item $(ii)$ follows using a test function
$\varphi\left(  x,t\right)  $ which takes the value $1$ for $1\leq |x|\leq M,$
combined with the fact that $\zeta_{M}\leq1$ and that $K\left(  x,y\right)
\geq a_{1}>0$ in the support of the measure $f\left(  dx,t\right)  f\left(
dy,t\right)  .$

The statements in item $(iii)$  about the space $\mathcal{X}_{M}$ 
and its intersections with 
norm-topology balls,  $\mathcal{U}_{R}$, 
follow from standard results in functional analysis, in particular, 
by using the Banach--Alaoglu theorem; more details may be found in \cite{FLNV}.
The properties of $S\left(  t\right)  $ in
(\ref{SemProps}) follow from the definition of $S\left(  t\right)  $ and the
uniqueness result in Proposition \eqref{thm:existence_evolution}. The only
nontrivial results to be proven are the continuity properties of $S\left(
t\right)  .$ The continuity of $t\mapsto S\left(  t\right)  f$ follows from
the differentiability of $f$ in the $t$ variable, yielding even Lipschitz-type 
estimates for the dependence, cf.\ Eq.\ (3.20) in \cite{FLNV}.  

The continuity of $f\mapsto S\left(  t\right)  f$ in 
the weak topology of measures is the most involved part of the results but it
can also be proven as in \cite{FLNV}. 
The basic idea is to prove that the mappings $f_{0}%
\mapsto\int_{\mathbb{R}_{*}^{d}}\varphi\left(  x\right)  f\left(
dx,t\right)  $ change continuously for every test function $\varphi\in
C_{c}\left(  \mathbb{R}_{*}^{d}\right)  .$ Considering the evolution of
$\varphi$ in terms of the adjoint equation from the time $t$ to the time $0$
we obtain a new function, denoted as $\varphi_{0}\in C_{c}\left(
\mathbb{R}_{*}^{d}\right)  $, such that 
\[
\int_{\mathbb{R}_{*}^{d}}\varphi\left(  x\right)  f\left(  dx,t\right)
=\int_{\mathbb{R}_{*}^{d}}\varphi_{0}\left(  x\right)  f_{0}\left(  dx\right)
.
\]
Therefore, $\int_{\mathbb{R}_{*}^{d}}\varphi\left( x\right)  f\left(
dx,t\right)  $ changes continuously if $\int_{\mathbb{R}_{*}^{d}}\varphi
_{0}\left( x\right)  f_{0}\left(  dx\right)  $ does. This gives the
desired continuity in the weak topology of measures and the result follows.  
\end{proof}

We can now prove the existence of stationary injection solutions to
(\ref{evolEqTrunc}) in the sense of Definition \ref{DefStatSol}.  We observe that, to  prove the existence  of stationary solutions, finding fixed points for the corresponding evolution semigroup has often been used in the study of coagulation equations. We refer for instance to \cite{EMR, FL05, NV, NNTV, NTV}. Similar ideas have been used also to construct stationary solutions of more general classes of kinetic equations. See for instance \cite{GPV, KV, JNV}.

\begin{proposition}
\label{thm:existence_truncated} Under the assumptions of
Proposition~\ref{thm:existence_evolution}, there exists a stationary injection
solution $\hat{f}\in\mathcal{M}_{+,b}(\mathbb{R}_{*}^{d})$ to (\ref{evolEqTrunc}%
) as in Definition~\ref{DefStatSol}.
\end{proposition}

\begin{proof}
The argument is analogous to the one in \cite{FLNV}, and we again merely provide a sketch
of the main details. The key idea is to construct an invariant region in the
space $\mathcal{X}_{M}$ under the evolution semigroup $S\left(  t\right)  .$
To this end, notice that (\ref{IneqQuad}) implies that for $R\geq\sqrt
{\frac{2\int_{\mathbb{R}_{*}^{d}}\eta\left(  dx\right)  }{a_{1}}}$ the set
\begin{equation}
\mathcal{U}_{R}=\left\{  f\in\mathcal{X}_{M}:\int_{{\mathbb{R}}_{\ast}^{d}%
}f(dx)\leq R\right\}  \label{eq:invariant_region}%
\end{equation}
is invariant under the evolution semigroup $S\left(  t\right)  ,$ i.e.,
$S\left(  t\right)  \left(  \mathcal{U}_{R}\right)  \subset\mathcal{U}_{R}.$ 
By Proposition \ref{PropSolutTimeDep},
the set $\mathcal{U}_{R}$ is convex and compact in the weak topology of
measures. Since the operator $S\left(  \delta\right)  :\mathcal{U}%
_{R}\rightarrow\mathcal{U}_{R}$ is continuous in the same topology,
it follows from Schauder's fixed point theorem that there
exists a fixed point $\hat{f}_{\delta}$ for each $\delta>0.$ 
Since $\mathcal{U}_{R}$ is metrizable and hence sequentially compact, 
we can use Theorem 1.2 of \cite{EMR}, and conclude that 
that there is $\hat{f}$ such that 
$S(t)\hat{f} = \hat{f}$ for all $t$.
Thus $\hat{f}$ is a stationary injection solution to
\eqref{evolEqTrunc}.
\end{proof}

We can now prove Theorem \ref{ThmExistCont}.

\bigskip

\begin{proofof}
[Proof of Theorem \ref{ThmExistCont}]Due to Lemma \ref{EquivDefin}, item
$(i)$, it is enough to prove the result for $p=0$ and $\gamma<1.$ Therefore,
we will restrict our attention to kernels of the form 
\begin{equation}
K\left(  x,y\right)  =\left(  \left\vert x\right\vert +\left\vert y\right\vert
\right)  ^{\gamma}\Phi\left(  x,y\right)  \label{KernType}%
\end{equation}
where $\Phi$ is continuous, symmetric, and
\begin{equation}
0<C_{1}\leq\Phi\left(  x,y\right)  \leq C_{2}<\infty\ \ ,\ \ \text{for }\;
x,y\in\mathbb{R}_{*}^{d} .\label{eq:B1bound}%
\end{equation}
We recall that we do not assume that the kernels $K\left(  x,y\right) $ are
homogeneous for this result. 

We define a class of truncated kernels by means of%
\begin{equation}
K_{\varepsilon}\left(  x,y\right)  =\min\left\{  \left(  \left\vert
x\right\vert +\left\vert y\right\vert \right)  ^{\gamma},\frac{1}{\varepsilon
}\right\}  \Phi\left(  x,y\right)  +\varepsilon\ \ ,\ \ \varepsilon>0
\label{eq:1levTrunc}%
\end{equation}
The kernels $K_{\varepsilon}$ are bounded in $\mathbb{R}_{*}^{d}%
\times\mathbb{R}_{*}^{d}$ for each $\varepsilon>0.$ Moreover, they satisfy
also $K_{\varepsilon}\left(  x,y\right)  \geq\varepsilon>0$ for any $\left(
x,y\right)  \in\mathbb{R}_{*}^{d}\times\mathbb{R}_{*}^{d}.$ 
We will assume in the following that $\varepsilon\le 1$, and add further restrictions to
its upper bound later on.
Note that if
$\gamma\leq0$, the truncation by the minimum in (\ref{eq:1levTrunc}) will not 
have any effect
since we are interested in solutions supported in $\left\vert x\right\vert
\geq1$ and $\left\vert y\right\vert \geq1.$

We now introduce another truncation to obtain compactly supported kernels. 
To this end we choose a symmetric function $o_{M}\in C_{c}(\mathbb{R}_{*}\times\mathbb{R}_{*})$  such
that for $x,y\in \mathbb{R}_{*}^{d}$, $r,s>0$,
\begin{equation}
\label{omegcutof}
\omega_M(x,y):=o_M(|x|,|y|)\,, \qquad
o_{M}(r,s)=
\begin{cases}
 1\,, & \text{if }
1\le r,s\le M\,,\\
 0\,, & \text{if }
 r \geq 2 M \text{ or } s \geq 2 M\,.
\end{cases}
\end{equation}
Note that then $\omega_{M}\in C_{c}(\mathbb{R}_{*}^{d}\times\mathbb{R}_{*}^{d})$ and it is symmetric.
We then define truncated kernels $K_{\varepsilon,M}$ by means of%
\begin{equation}
K_{\varepsilon,M}\left(  x,y\right)  =K_{\varepsilon}\left(  x,y\right)
\omega_{M}\left(  x,y\right) \,, \qquad  x,y\in \mathbb{R}_{*}^{d}\,. \label{eq:2levTrunc}%
\end{equation}
Each of these kernels satisfies the requirements of $K_M$ in Assumption \ref{Assumptions}.

We finally choose a function $p_{M}\in C(\R_*)$ satisfying Assumption
\ref{Assumptions} for $d=1$.  Then defining $\zeta_{M}\left(
x\right)  := p_{M}\left(  \left\vert x\right\vert \right)$ satisfies Assumption
\ref{Assumptions} for general $d$.
The hypothesis of Theorem \ref{ThmExistCont} imply that $\eta$ satisfies the
requirements for this function in Assumption \ref{Assumptions}. 
Therefore, the above choices result in a system which satisfies all 
conditions required in Assumption \ref{Assumptions}, and thus
Proposition~\ref{thm:existence_truncated} implies the existence of a
stationary injection solution $f_{\varepsilon,M}\in{\mathcal{M}_{+,b}(\mathbb{R}_{*}^{d})}$ satisfying $f_{\varepsilon,M}\left(
\left\{  x\in{\mathbb{R}_{*}^{d}}:|x|<1\text{ or }\left\vert x\right\vert
>M\right\}  \right)  =0$ and 
\begin{align}
&  \frac{1}{2}\int_{{\mathbb{R}_{*}^{d}}}\int_{{\mathbb{R}_{*}^{d}}%
}K_{\varepsilon,M}\left(  x,y\right)  \left[  \varphi\left(  x+y\right)
\zeta_{M}\left(  x+y\right)  -\varphi\left(  x\right)  -\varphi\left(
y\right)  \right]  f_{\varepsilon,M}\left(  dx\right)  f_{\varepsilon
,M}\left(  dy\right) \label{WeakTruncForm}\\
&  +\int_{{\mathbb{R}_{*}^{d}}}\varphi\left(  x\right)  \eta\left(
dx\right)  =0\nonumber
\end{align}
for any test function $\varphi\in{C}_{c}({\mathbb{R}}^d_{*})$.

We now derive uniform estimates for the family of solutions $f_{\varepsilon
,M}$ in order to take the limits $M\rightarrow\infty$ and then $\varepsilon
\rightarrow0.$  To this end, it will be convenient to use the
coordinates $\left(  r,\theta\right)  $ introduced in \eqref{S1E1} and defined
as $r=\left\vert x\right\vert ,\ \theta=\frac{x}{\left\vert x\right\vert }.$
We write also $\rho=\left\vert y\right\vert ,\ \sigma=\frac{y}{\left\vert
y\right\vert }$ as well as $\varphi\left(  x\right)  =\psi\left(
r,\theta\right)$, $K_{\varepsilon,M}\left(  x,y\right)
=G_{\varepsilon,M}\left(  r,\rho;\theta,\sigma\right)$,
and $f_{\varepsilon,M}\left(  x\right)  =F_{\varepsilon,M}\left(  r,\theta\right)$,
as defined in (\ref{G2}).
Then
(\ref{WeakTruncForm}) becomes%
\begin{align}
&  \frac{1}{2 d}\int_{\mathbb{R}_{*}}r^{d-1}dr\int_{\mathbb{R}_{*}}\rho^{d-1}d\rho
\int_{\Delta^{d-1}}d\tau\left(  \theta\right)  \int_{\Delta^{d-1}}d\tau\left(
\sigma\right)  G_{\varepsilon,M}\left(  r,\rho;\theta,\sigma\right)
F_{\varepsilon,M}\left(  r,\theta\right)  F_{\varepsilon,M}\left(  \rho
,\sigma\right)  \times\nonumber\\
&  \quad\times\left[ p_M(r+\rho) \psi\left(  r+\rho,\frac{r}{r+\rho}\theta+\frac{\rho
}{r+\rho}\sigma\right)  -\psi\left(  r,\theta\right)  -\psi\left(  \rho
,\sigma\right)  \right]  +\int_{{\mathbb{R}_{*}^{d}}}\varphi\left(
x\right)  \eta\left(  dx\right)  =0. \label{WeakTrunc}%
\end{align}

Given $z,\delta >0$, we introduce a function 
$\chi_{\delta}\in C^\infty_c({\mathbb{R}}_{*})$ such that $0\le \chi_{\delta}\le 1$, 
$\chi_{\delta}(s)=1$ for $1\le \left\vert s\right\vert \leq z$, and $\chi_{\delta}(s)=0$ for $\left\vert s\right\vert \geq z+\delta$.
 We then take the radial test function $\varphi
(x)=\left\vert x\right\vert \chi_{\delta}\left(  \left\vert x\right\vert
\right)$ for $x\in \R^d_*$.  Then $\psi\left(  r,\theta\right)  =r\chi_{\delta}\left(
r\right)  .$ Plugging this test function into (\ref{WeakTrunc}), and using $p_M\le 1$, we obtain%
\begin{align*}
&  \int_{\mathbb{R}_{*}}r^{d-1}dr\int_{\mathbb{R}_{*}}\rho^{d-1}d\rho
\int_{\Delta^{d-1}}d\tau\left(  \theta\right)  \int_{\Delta^{d-1}}d\tau\left(
\sigma\right)  G_{\varepsilon,M}\left(  r,\rho;\theta,\sigma\right)
F_{\varepsilon,M}\left(  r,\theta\right)  F_{\varepsilon,M}\left(  \rho
,\sigma\right)  \times\\
&  \quad\times \left[  \left(  r+\rho\right)  \chi_{\delta}\left(  r+\rho\right)
-r\chi_{\delta}\left(  r\right)  -\rho\chi_{\delta}\left(  \rho\right)
\right]  +2 d\int_{{\mathbb{R}_{*}^{d}}}\left\vert x\right\vert \chi_{\delta
}\left(  \left\vert x\right\vert \right)  \eta\left(  dx\right)  \ge 0.
\end{align*}
Taking the limit $\delta\rightarrow0$ and using arguments analogous to
the ones in the proof of Lemma 2.7 in \cite{FLNV}, yields
\begin{align}
& \frac{1}{d} \int_{\left(  0,z\right]  }r^{d}dr\int_{\left(  z-r,\infty\right)  }%
\rho^{d-1}d\rho\int_{\Delta^{d-1}} d\tau\left(  \theta\right)  \int_{\Delta^{d-1}} d\tau\left(  \sigma\right)  G_{\varepsilon,M}\left(  r,\rho;\theta,\sigma\right)  F_{\varepsilon
,M}\left(  r,\theta\right)  F_{\varepsilon,M}\left(  \rho,\sigma\right)
\label{FluxTruncMulti}  \nonumber \\
&  \le \int_{\left\{ 0< \left\vert x\right\vert \leq z\right\}  }\left\vert
x\right\vert \eta\left(  dx\right) \,, \qquad \text{for any }z>0\,.
\end{align}

We can now derive uniform estimates for the measures
$F_{\varepsilon,M}$ arguing as in \cite{FLNV}. We have $d \int_{\left\{
\left\vert x\right\vert \leq z\right\}  }\left\vert x\right\vert \eta\left(
dx\right)  \leq d\int_{\left\{  \left\vert x\right\vert \leq L\right\}
}\left\vert x\right\vert \eta\left(  dx\right) =: c.$ On the other hand
(\ref{eq:1levTrunc}), (\ref{eq:2levTrunc}) imply that $K_{\varepsilon,M
}\left(  x,y\right)  =G_{\varepsilon,M}\left(  r,\rho;\theta,\sigma\right)
\geq\varepsilon>0$ for $\left(  r,\rho\right)  \in\left[  1,M\right]  ^{2}.$
Thus we may use the information about the support of $f_{\vep,M}$ to conclude that
for any $z>0$
\[
\varepsilon\int_{\left(  0,z\right]  }r^{d}dr\int_{\left(  z-r,\infty\right)
}\rho^{d-1}d\rho \int_{\Delta^{d-1}} d\tau\left(  \theta\right)  \int_{\Delta^{d-1}} d\tau\left(  \sigma\right)  
F_{\varepsilon,M}\left(  r,\theta\right)  F_{\varepsilon,M}\left(
\rho,\sigma\right)  \leq c\,.
\]

Assume next $z\in [1,M]$. Then $[2z/3,z]^{2}\subset\left\{  \left(  x,y\right)
\in\mathbb{R}_{+}^{2}:0<x\leq z,\ z-x<y\leq M\right\}  $, and we obtain
\[
\varepsilon z^{2d-1}\left(  \int_{\left[  \frac{2z}{3},z\right]  }%
dr\int_{\Delta^{d-1}}d\tau\left(  \theta\right)  F_{\varepsilon,M}\left(  r,\theta\right)  \right)  ^{2}\leq C\,.
\]
Therefore,
\[
\frac{1}{z}\int_{\left[  \frac{2z}{3},z\right]  }dr\int_{\Delta^{d-1}}%
d\tau\left(  \theta\right)  F_{\varepsilon,M}\left(
r,\theta\right)  \leq\frac{C_{\varepsilon}}{z^{\frac{2d+1}{2}}}\ \ \ \text{for
}0<z\leq M
\]
where $C_{\varepsilon}$ is a numerical constant depending on $\varepsilon$ but
independent of $M$. Then, Lemma \ref{lem:bound} implies%
\begin{equation}
\int_{\R^d_*}f_{\varepsilon,M}\left(  dx\right) =
\int_{\left\{  1\le \left\vert x\right\vert \leq M\right\}
}f_{\varepsilon,M}\left(  dx\right) =\int_{[1,M]}r^{d-1}%
dr\int_{\Delta^{d-1}} 
d\tau\left(  \theta\right) F_{\varepsilon,M}\left(  r,\theta\right)  \leq\bar{C}_{\varepsilon
}\,.\label{S2E8}%
\end{equation}
with $\bar{C}_{\varepsilon}$ independent on $M.$

By the earlier mentioned sequential compactness, this bound implies that 
for each $\varepsilon>0$ there exists a sequence $\left\{
M_{n}\right\}  _{n\in\mathbb{N}}$ with $\lim_{n\rightarrow\infty}M_{n}=\infty$
and $f_{\varepsilon}\in\mathcal{M}_{+,b}\left(  \mathbb{R}_{*}^{d}\right)  $
such that $f_{\varepsilon,M_{n}}\rightarrow f_{\varepsilon}$ in the weak
topology of measures. Moreover, we have%
\[
\int_{\mathbb{R}_{*}^{d}}f_{\varepsilon}\left(  dx\right)  \leq\bar
{C}_{\varepsilon}%
\]
and also $f_{\varepsilon}\left(
\left\{  x\in{\mathbb{R}_{*}^{d}}:|x|<1\right\}  \right)  =0$. 

We now notice
that $\zeta_{M_{n}}\left(  x+y\right)  \rightarrow1$ as $n\rightarrow\infty$
if $\left\vert x+y\right\vert >0$. Using this, as well as (\ref{S2E8}), we can now take
the limit $n\rightarrow\infty$ in (\ref{WeakTruncForm}) with $M$ replaced by $M_{n}$
(more details about these estimates can be found from the proof of Theorem 2.3 in \cite{FLNV}).
Since $K_{\varepsilon}(x,y)=\lim_{M\rightarrow\infty
}K_{\varepsilon,M}(x,y)$, and using that $K_{\varepsilon}$ is bounded for each
$\varepsilon>0$, we then obtain 
\begin{equation}
\frac{1}{2}\int_{{\mathbb{R}_{*}^{d}}}\int_{{\mathbb{R}_{*}^{d}}%
}K_{\varepsilon}\left(  x,y\right)  \left[  \varphi\left(  x+y\right)
-\varphi\left(  x\right)  -\varphi\left(  y\right)  \right]  f_{\varepsilon
}\left(  dx\right)  f_{\varepsilon}\left(  dy\right)  +\int_{{{\mathbb{R}}%
_{*}^{d}}}\varphi\left(  x\right)  \eta\left(  dx\right)  =0\label{S2E9}%
\end{equation}
for any test function $\varphi\in C_{c}\left(  {\mathbb{R}_{*}^{d}}\right)
.$ We now argue again as in the derivation of (\ref{FluxTruncMulti}), using
the test function $\varphi(x)=\left\vert x\right\vert \chi_{\delta}\left(
\left\vert x\right\vert \right)  ,$ or equivalently $\psi\left(
r,\theta\right)  =r\chi_{\delta}\left(  r\right)  .$ We define $F_{\varepsilon
}\left(  r,\theta\right)  =f_{\varepsilon}\left(  x\right)$ and $K_{\varepsilon}\left(  x,y\right)
=G_{\varepsilon}\left(  r,\rho;\theta,\sigma\right).$  Then, taking the limit $\delta\rightarrow0$ we arrive at
\begin{align*}
&  \frac{1}{d}\int_{\left(  0,z\right]  }r^{d}dr\int_{\left(  z-r,\infty\right)  }%
\rho^{d-1}d\rho\int_{\Delta^{d-1}} d\tau\left(  \theta\right)  \int_{\Delta^{d-1}} d\tau\left(  \sigma\right)  G_{\varepsilon}\left(  r,\rho;\theta,\sigma\right)  F_{\varepsilon}\left(  r,\theta\right)  F_{\varepsilon}\left(  \rho,\sigma\right)
\nonumber \\
&  = \int_{\left\{ 0< \left\vert x\right\vert \leq z\right\}  }\left\vert
x\right\vert \eta\left(  dx\right) \,, \qquad \text{for any }z>0\,.
\end{align*}

Using (\ref{eq:B1bound}) and
(\ref{eq:1levTrunc}) we have $K_{\varepsilon}(x,y)\geq\varepsilon+C_{0}%
\min\{z^{\gamma},\frac{1}{\varepsilon}\}$ for $\left\vert x\right\vert
,\left\vert y\right\vert \in\left[  \frac{z}{2},z\right]  $ with $C_{0}>0.$
Using then that $[2z/3,z]^{2}\subset\left\{  \left(  r,\rho\right)
\in\mathbb{R}_{+}^{2}:0<r\leq z,\ z-x<y<\infty\right\}  $ we obtain%
\[
\left(  \varepsilon+\min\{z^{\gamma},\frac{1}{\varepsilon}\}\right)  z\left(
\int_{[2z/3,z]}r^{d-1}dr\int_{\Delta^{d-1}}F_{\varepsilon}\left(
r,\theta\right)  d\tau\left(  \theta\right)  \right)  ^{2}\leq c \int_{\R_{*}^d} \vert x\vert \eta(dx)\ \text{ for all
}z>0 \,,
\]
where $c$ is independent of $\varepsilon$ and $\eta$. Thus%
\begin{align}
\frac{1}{z}\int_{[2z/3,z]}r^{d-1}dr\int_{\Delta^{d-1}}F_{\varepsilon}\left(
r,\theta\right)  d\tau\left(  \theta\right)  &\leq\frac{\tilde{C}}{z^{\frac
{3}{2}}\left(  \varepsilon+\min\{z^{\gamma},\frac{1}{\varepsilon}\}\right)
^{\frac{1}{2}}}\left( \int_{\R_{*}^d} \vert x\vert \eta(dx) \right)^{\frac 1 2}\nonumber \\& \leq\frac{\tilde{C}}{z^{\frac{3}{2}}\left(  \min\{z^{\gamma
},\frac{1}{\varepsilon}\}\right)  ^{\frac{1}{2}}} \left( \int_{\R_{*}^d} \vert x\vert \eta(dx)\right)^{\frac 1 2} \label{A1N}%
\end{align}
with $\tilde{C}$ independent of $\varepsilon$ and $\eta$, for $z>0$.  If $z\ge 1$, we have $ \min\{z^{\gamma
},\frac{1}{\varepsilon}\}\ge z^{\gamma_-}$ where $\gamma_-:= \min(0,\gamma)\le 0$.  Thus we can conclude 
that
\begin{align}
\frac{1}{z}\int_{[2z/3,z]}r^{d-1}dr\int_{\Delta^{d-1}}F_{\varepsilon}\left(
r,\theta\right)  d\tau\left(  \theta\right) 
 \leq\frac{\tilde{C}}{z^{\frac{3}{2}+\gamma_-}} \left( \int_{\R_{*}^d} \vert x\vert \eta(dx)\right)^{\frac 1 2}
 \,,\quad \text{for all }z\ge 1\,. \label{A1Nbis}%
\end{align}
This
implies the existence of a subsequence 
$\left\{  \varepsilon_{n}\right\}_{n\in\mathbb{N}}$ 
with $\lim_{n\rightarrow\infty}\varepsilon_{n}=0$ 
such that we have 
$|x|^{\gamma_-}f_{\varepsilon_{n}}(d x)\rightarrow |x|^{\gamma_-}f(dx)$ in the weak measure topology of $\mathcal{M}%
_{+,b}\left(  \mathbb{R}_{*}^{d}\right)$.  Clearly, then 
$f\in \mathcal{M}_{+}\left(  \mathbb{R}_{*}^{d}\right)$ and also
$f\left(
\left\{  x\in{\mathbb{R}_{*}^{d}}:|x|<1\right\}  \right)  =0$.

It only remains to check that we can
take the limit $n\rightarrow\infty$ in (\ref{S2E9}) with $\varepsilon
=\varepsilon_{n}$ for any test function $\varphi\in C_{c}\left(
\mathbb{R}_{*}^{d}\right)  .$ We can readily take the limit $n\rightarrow
\infty$, using (\ref{A1N}), in the term $\int_{{\mathbb{R}_{*}^{d}}}\int_{{\mathbb{R}_{*}^{d}%
}}K_{\varepsilon_{n}}\left(  x,y\right)  \varphi\left(  x+y\right)
f_{\varepsilon_{n}}\left(  dx\right)  f_{\varepsilon_{n}}\left(  dy\right)  $
because $\varphi$ is compactly supported and hence the integration
may be restricted to a compact subset of ${\mathbb{R}_{*}^{d}\times{\mathbb{R}}_{*}^{d}}$. It only remains
to examine the terms of (\ref{S2E9}) which contain the functions
$\varphi\left(  x\right)  ,\ \varphi\left(  y\right).$ These contributions are
identical, as it can be seen exchanging the variables $x$ and $y.$ Therefore, it
is enough to consider only one of them, say $\int_{{\mathbb{R}_{*}^{d}}}%
\int_{{\mathbb{R}_{*}^{d}}}K_{\varepsilon_{n}}\left(  x,y\right)
\varphi\left(  x\right)  f_{\varepsilon_{n}}\left(  dx\right)  f_{\varepsilon
_{n}}\left(  dy\right)  .$ Since $\varphi$ is compactly supported we only need
to obtain uniform boundedness of $\int_{{\mathbb{R}_{*}^{d}}}K_{\varepsilon
_{n}}\left(  x,y\right)  f_{\varepsilon_{n}}\left(  dy\right)  $ for $x$
bounded and for $|y|$ large. Due to (\ref{eq:1levTrunc}), (\ref{A1N}), as well as Lemma
\ref{lem:bound}, it is enough to estimate%
\[
\int_{1}^{\infty}\frac{\min\left\{  \rho^{\gamma},\frac{1}{\varepsilon
}\right\}  +\varepsilon}{\rho^{\frac{3}{2}}\left(  \varepsilon+\min
\{\rho^{\gamma},\frac{1}{\varepsilon}\}\right)  ^{\frac{1}{2}}}d\rho.
\]

It is readily seen that this integral can be bounded (up to a multiplicative
constant) by the sum%
\begin{equation}
\int_{1}^{\infty}\frac{\min\left\{  \rho^{\gamma},\frac{1}{\varepsilon
}\right\}  }{\rho^{\frac{3}{2}}\left(  \min\{\rho^{\gamma},\frac
{1}{\varepsilon}\}\right)  ^{\frac{1}{2}}}d\rho+\int_{1}^{\infty}%
\frac{\varepsilon}{\rho^{\frac{3}{2}}\left(  \varepsilon\right)  ^{\frac{1}
{2}}}d\rho \ .\label{S3E1}%
\end{equation}
The second integral in (\ref{S3E1}) is given by $2\left(  \varepsilon\right)
^{\frac{1}{2}}$ and thus it tends to zero as $\varepsilon\rightarrow0.$ On the
other hand, the first integral in (\ref{S3E1}) can be bounded as%
\[
\int_{1}^{\infty}\frac{\left(  \min\{\rho^{\gamma},\frac{1}{\varepsilon
}\}\right)  ^{\frac{1}{2}}d\rho}{\rho^{\frac{3}{2}}}\leq\int_{1}^{\infty}%
\frac{d\rho}{\rho^{\frac{3-\gamma}{2}}}<\infty 
\]
since $\gamma<1.$ This gives the desired uniform estimate. We can then take the limit $n\rightarrow\infty$ in
(\ref{S2E9}). 

It remains to check that also $\int_{\R^d_*} |x|^{\gamma} f(dx)
<\infty$ which is obvious from the construction, if $\gamma\le 0$.  If $0<\gamma<1$,
this follows by first taking $\vep\to 0$ in (\ref{A1N}) and then using Lemma \ref{lem:bound}.
Therefore, $f$ is a solution to (\ref{eq:contStat}) in the sense
of Definition \ref{StatInjDef}. This concludes the proof of Theorem
\ref{ThmExistCont}.
\end{proofof}

\subsection{Discrete coagulation equation}

We now prove Theorem \ref{ThmExistDisc}. Given that the proof is very similar
to the one for the continuous case (cf. Subsection \ref{ExistContCase}) we just sketch the main ideas. 

\bigskip

\begin{proofof}[Proof of Theorem \ref{ThmExistDisc}]
Due to item $(ii)$ of Lemma \ref{EquivDefin} it is enough to prove the result for
$p=0,\ \gamma<1.$ By assumption the kernel $K$ can be written as in
(\ref{KernType}) with $\Phi$ satisfying (\ref{eq:B1bound}) for all
$x,y\in\mathbb{N}_{*}^{d}  .$ We truncate the
kernel $K$ as in the proof of Theorem \ref{ThmExistCont} (cf.
(\ref{eq:1levTrunc}), (\ref{omegcutof}) and (\ref{eq:2levTrunc})) with
$\zeta_{M}$ chosen as explained in the paragraph after (\ref{eq:2levTrunc}). 
Then $\zeta_{M}$ satisfies 
Assumption \ref{Assumptions}. 

We then consider the following
time dependent truncated problem, where $M>2L$,%
\begin{equation}
\partial_{t}n_{\alpha}=\frac{\zeta_{M}\left(  \alpha\right)  }{2}\sum
_{\beta<\alpha}K_{\varepsilon,M}\left(  \alpha-\beta,\beta\right)
n_{\alpha-\beta}n_{\beta}-n_{\alpha}\sum_{\beta>0}K_{\varepsilon,M}\left(
\alpha,\beta\right)  n_{\beta}+ s_{\alpha}, \; \; \alpha \in\mathbb{N}_{*}^{d}.%\ \
 \label{DiscTrunc}%
\end{equation}
We can construct solutions of (\ref{DiscTrunc}) satisfying $n_{\alpha}=0$ for
$\left\vert \alpha\right\vert \geq 2M$ for any nonnegative initial distribution
$n_{0,\alpha}$ satisfying the same property. Local existence follows from
classical ODE theory. Global existence follows from the estimate%
\begin{equation}
\partial_{t}\left(  \sum_{\alpha}n_{\alpha}\right)  \leq-\frac{\varepsilon}%
{2}\left(  \sum_{\alpha}n_{\alpha}\right)  ^{2}+\sum_{\alpha}s_{\alpha}
\label{InvDiscTrunc}%
\end{equation}
that implies 
$$\sum_{\alpha}n_{\alpha}\leq\sum_{\alpha}n_{0,\alpha}%
+t\sum_{\alpha}s_{\alpha}.$$  
Due to (\ref{InvDiscTrunc}) there exists an invariant convex set defined by means of 
$$\left\{  n_{\alpha}%
:\sum_{\alpha}n_{\alpha}\leq\sqrt{\frac{2\sum_{\alpha}s_{\alpha}}{\varepsilon
}},\ n_{\alpha}\geq0\right\}.$$ Therefore, the existence of a stationary solution is
a consequence of Schauder's Fixed Point Theorem, arguing as in the proof of
Proposition \ref{thm:existence_truncated}. This solution will be denoted as
$\left\{  n_{\alpha}^{\varepsilon,M}\right\}  _{\alpha\in\mathbb{N}_{0}%
^{d}\setminus\left\{  O\right\}  }.$ 

We can now take the limit $M\rightarrow
\infty$ and then $\varepsilon\rightarrow0$ in order to obtain a stationary
injection solution to (\ref{S1E5}). To this end, we derive uniform
estimates for the sequence $\left\{  n_{\alpha}^{\varepsilon,M}\right\}
_{\alpha\in\mathbb{N}_{*}^{d}  }.$ 
More precisely, we have already proved 
the estimate $\sum_{\alpha}n_{\alpha}^{\varepsilon,M}\leq\sqrt{\frac
{2\sum_{\alpha}s_{\alpha}}{\varepsilon}}.$ Since the right-hand side is
independent of $M,$ there exists a sequence $\left\{  M_{n}\right\}
_{n\in\mathbb{N}}$ such that $M_{n}\rightarrow\infty$ as $n\rightarrow\infty$
and $n_{\alpha}^{\varepsilon,M_{n}}\rightarrow n_{\alpha}^{\varepsilon}$ as
$n\rightarrow\infty,$ where the sequence $\left\{  n_{\alpha}^{\varepsilon
}\right\}  _{\alpha\in\mathbb{N}_{*}^{d}  }$ solves%
\begin{equation}
\frac{1}{2}\sum_{\beta<\alpha}K_{\varepsilon}\left(  \alpha-\beta
,\beta\right)  n_{\alpha-\beta}^{\varepsilon}n_{\beta}^{\varepsilon}%
-n_{\alpha}^{\varepsilon}\sum_{\beta>0}K_{\varepsilon}\left(  \alpha
,\beta\right)  n_{\beta}^{\varepsilon}+ s_{\alpha }=0\ \ ,\ \ \alpha\in\mathbb{N}_{*}^{d}.
\label{TruncDisc}%
\end{equation}

In order to estimate the sequence $\left\{  n_{\alpha}^{\varepsilon}\right\}
_{\alpha\in\mathbb{N}_{*}^{d}  }$ we use the fact
that the measure $f_{\varepsilon}=\sum_{\alpha}n_{\alpha}^{\varepsilon}%
\delta\left(  \cdot-\alpha\right)  $ solves a stationary continuous equation.
More precisely, $f_{\varepsilon}$ satisfies (\ref{S2E9}) for any test function
$\varphi\in C_{c}\left(  {\mathbb{R}_{*}^{d}}\right)  ,$ where
$K_{\varepsilon}$ is any continuous extension of the discrete kernel. We can
then derive, arguing as in the proof of Theorem \ref{ThmExistCont} that
$f_{\varepsilon}$ satisfies the estimate (\ref{A1N}) with $f_{\varepsilon
}\left(  x\right)  =F_{\varepsilon}\left(  r,\theta\right)  .$ We can then
show that there exists a sequence $\left\{  \varepsilon_{n}\right\}
_{n\in\mathbb{N}}$ with $\lim_{n\rightarrow\infty}\varepsilon_{n}=0$ such that
$n_{\alpha}^{\varepsilon_{n}}\rightarrow n_{\alpha}$ as $n\rightarrow\infty$
for each $\alpha\in\mathbb{N}_{*}^{d}  .$ Moreover,
using (\ref{A1N}) we can also pass to the limit in the weak form of
(\ref{TruncDisc}) to show that $\left\{  n_{\alpha}\right\}  _{\alpha
\in\mathbb{N}_{*}^{d}  }$ is a stationary injection
solution to (\ref{S1E5}). Hence the Theorem follows.
\end{proofof}

\section{Nonexistence results}\label{sec:9}

In this Section we prove the non-existence of stationary injection solutions for the continuous and discrete 
model as well as the non-existence of constant flux solutions.

We first prove Theorem \ref{NonexistCont}.  
Due to Lemma \ref{EquivDefin}, item $(i)$, it is enough to prove the result for
$\gamma\geq1,\ p=0.$ We first recall the following auxiliary Lemma that is
a particular case of Lemma 4.1 in \cite{FLNV} with $a\geq1,\ b=0$.

\begin{lemma}
\label{lem:F+estimate} Let $a\geq1$ be a constant. Let
$W:\mathbb{R}_{*}\rightarrow{\mathbb{R}}$ be a right-continuous non-increasing
function satisfying $W(R)\geq0$, for all $R>0$. Assume that $h\in
\mathcal{M}_{+}(\mathbb{R}_{*})$ satisfies $h([1,\infty))>0$ and
\begin{equation}
\int_{\lbrack1,\infty)}x^{a}h(dx)<\infty\,. \label{eq:momentBound1}
\end{equation}
Suppose that there exists $\delta$ such that $0<\delta<1$ and the following
inequality holds
\begin{equation}
-\int_{\left[  1,\delta R\right]  }\left[  W\left(  R-y\right)  -W\left(
R\right)  \right]  h\left(  dy\right)  \leq -\frac{C}{R^{a+1}}\,,\quad\text{for
}R\geq R_{0}, \label{S4E9_1}%
\end{equation}
for some $R_{0}>1/\delta$ and $C>0$.

Then there are two constants $R_{0}^{\prime}\geq R_{0}$ and $B>0$ which depend only on $a$,
$h$, $\delta$, $R_{0}$, and $C$, such that
\begin{equation}
W\left(  R\right)  \geq\frac{B}{R^{a}}\,,\quad\text{for \ }R\geq R_{0}%
^{\prime}, \label{S5E3_1}%
\end{equation}

\end{lemma}

\bigskip

The proof of Lemma \ref{lem:F+estimate} relies on a comparison argument and on the construction of a suitable
subsolution for the problem \eqref{S4E9_1}.  \medskip

\begin{proofof}
[Proof of Theorem \ref{NonexistCont}] 
We start proving item $(i)$ which refers to the continuous case. To this end we follow the same strategy as in the proof of Theorem 2.4 in \cite{FLNV} for the one-component case. 
To get a contradiction, let us assume $f$ is a stationary injection solution.
Let $F$ be the corresponding measure in the simplex coordinates, as explained after Definition \ref{StatInjDef}.  In particular, then 
the support of $F$ 
lies in $[1,\infty)\times \Delta^{d-1}$ and $F$ satisfies
$
 \int_{\R_*\times \Delta^{d-1}} r^{d-1+\gamma} F(r,\theta)dr d\tau(\theta)
<\infty$.
We define $G$ as before from $K$.  Then, after rewriting (\ref{WeakForm1}) in the simplex coordinate system, we find that
\begin{align}
&  \frac{1}{2 d}\int_{\mathbb{R}_{*}}r^{d-1}dr\int_{\mathbb{R}_{*}}\rho
^{d-1}d\rho\int_{\Delta^{d-1}}d\tau\left(  \theta\right)  \int_{\Delta^{d-1}%
}d\tau\left(  \sigma\right)  G\left(  r,\rho;\theta,\sigma\right)  F\left(
r,\theta\right)  F\left(  \rho,\sigma\right) \times \nonumber\\
&  \times\left[  \psi\left(  r+\rho,\frac{r}{r+\rho}\theta+\frac{\rho}{r+\rho
}\sigma\right)  -\psi\left(  r,\theta\right)  -\psi\left(  \rho,\sigma\right)
\right]  +\int_{{\mathbb{R}_{*}^{d}}}\varphi\left(  x\right)  \eta\left(
dx\right) =0
\label{S8E6}%
\end{align}
for every $\psi \in C^{1}_c(\R_{*}\times \Delta^{d-1})$. 

We now choose test functions of the form $\psi\left(  r,\theta\right)  =r\chi_{R,\delta
}\left(  r\right)  $ to derive a formula for the fluxes. 
Let $R> \delta>0$.
We assume that
$\chi_{R,\delta}\in C^\infty_c(\R_*)$ is a ``bump function'', more precisely, it is monotone increasing on $(0,R]$ monotone decreasing on $[R,\infty)$. We also assume that  
$\chi_{R,\delta}\left(  s\right)  =1$ for $\delta\le s\leq R,\ \chi_{R,\delta}\left(
s\right)  =0$ for $s\geq R+\delta$.  
We then have%
\begin{align*}
 \psi\left(  r+\rho,\frac{r}{r+\rho}\theta+\frac{\rho}{r+\rho}\sigma\right)
&-\psi\left(  r,\theta\right)  -\psi\left(  \rho,\sigma\right) = \left(  r+\rho\right)  \chi_{R,\delta}\left(  r+\rho\right)
-r\chi_{R,\delta}\left(  r\right)  -\rho\chi_{R,\delta}\left(  \rho\right)  \\
&  =-r\left[  \chi_{R,\delta}\left(  r\right)  -\chi_{R,\delta}\left(
r+\rho\right)  \right]  -\rho\left[  \chi_{R,\delta}\left(  \rho\right)
-\chi_{R,\delta}\left(  r+\rho\right)  \right] \ .
\end{align*}
Plugging this identity in (\ref{S8E6}) and assuming that $R>L$ we obtain%
\begin{align}
&  |J| d=\int_{\mathbb{R}_{*}}r^{d-1}dr\int_{\mathbb{R}_{*}}\rho^{d-1}d\rho
\int_{\Delta^{d-1}}d\tau\left(  \theta\right)  \int_{\Delta^{d-1}}d\tau\left(
\sigma\right)  G\left(  r,\rho;\theta,\sigma\right)  F\left(  r,\theta\right)
F\left(  \rho,\sigma\right)  \nonumber \label{T1E1}\\
& \qquad \times r\left[  \chi_{R,\delta}\left(  r\right)  -\chi_{R,\delta}\left(
r+\rho\right)  \right] \,,
\end{align}
where $J=\int_{{\mathbb{R}_{*}^{d}}}x\eta\left(  dx\right)  \in
\mathbb{R}_{*}^{d}.$ Taking the limit $\delta\rightarrow0$ we may then conclude that for all $R>L$
\begin{equation}
\int_{\left(  0,R\right]  }r^{d-1}dr\int_{\left(  R-r,\infty\right)  }%
\rho^{d-1}d\rho\int_{\Delta^{d-1}}d\tau\left(  \theta\right)  \int
_{\Delta^{d-1}}d\tau\left(  \sigma\right)  G\left(  r,\rho,\theta
,\sigma\right)  F\left(  r,\theta\right)  F\left(  \rho,\sigma\right)
r=\left\vert J\right\vert d .\label{S8E7}%
\end{equation}

By the earlier mentioned known properties of $F$,
\begin{equation}
\int_{\left[  1,\infty\right)  }\rho^{d-1+\gamma}d\rho\int_{\Delta^{d-1}}%
d\tau\left(  \sigma\right)  F\left(  \rho,\sigma\right)  <\infty\,.\label{S8E8}%
\end{equation}
We estimate first the contribution to the integral in (\ref{S8E7}) due to the
region $\left\{  \rho\geq\delta r\right\}  $ where $0<\delta< \frac{1}{2}$. Using the upper estimate in
\eqref{Ineq2} (see also \eqref{S2E2}) we obtain%
\begin{align}
&  \int_{\left\{  \left(  r,\rho\right)  :r\in\left( 0,R\right]
,\ r+\rho> R,\ \rho\geq\delta r\right\}  }r^{d-1}\rho^{d-1}drd\rho
\int_{\Delta^{d-1}}d\tau\left(  \theta\right)  \int_{\Delta^{d-1}}d\tau\left(
\sigma\right)  G\left(  r,\rho;\theta,\sigma\right)  F\left(  r,\theta\right)
F\left(  \rho,\sigma\right)  r\nonumber\\
&  \leq C_{\delta}\int_{\left[  1,R\right]  }r^{d-1+1}dr\int_{\left[
\frac{\delta R}{1+\delta},\infty\right)  }\rho^{d-1}d\rho\int_{\Delta^{d-1}%
}d\tau\left(  \theta\right)  \int_{\Delta^{d-1}}d\tau\left(  \sigma\right)
\rho^{\gamma}F\left(  r,\theta\right)  F\left(  \rho,\sigma\right)
\nonumber\\
&  \leq C_{\delta}\left(  \int_{\left[  1,R\right]  }r^{d-1+\gamma}%
dr\int_{\Delta^{d-1}}F\left(  r,\theta\right)  d\tau\left(  \theta\right)
\right)  \int_{\left[  \frac{\delta R}{1+\delta},\infty\right)  }%
\rho^{d-1+\gamma}d\rho\int_{\Delta^{d-1}}F\left(  \rho,\sigma\right)
d\tau\left(  \sigma\right)  \label{S8E9}%
\end{align}
where $\gamma\geq1$, $p=0$. It then follows from \eqref{S8E8} that the right-hand side
of (\ref{S8E9}) tends to zero as $R\rightarrow\infty.$ 

Therefore, we can now conclude from  (\ref{S8E7}) that for every $\delta\in (0,1)$ there is $R_\delta>L,\delta^{-1}$ such that, if $R\ge R_\delta$, then 
\begin{align*}
& \int_{\left\{  \left(  r,\rho\right)  :\, r\in\left( 0,R\right],\, r+\rho>
R,\, \rho<\delta r\right\}  } r^{d}\rho^{d-1}drd\rho
\int_{\Delta^{d-1}}\!
d\tau\left(  \theta\right)  \int_{\Delta^{d-1}}\! d\tau\left(  \sigma\right)
G\left(  r,\rho;\theta,\sigma\right)  F\left(  r,\theta\right)  F\left(
\rho,\sigma\right)
\\ & \quad 
\geq\frac{\left\vert J\right\vert d}{2}.  
\end{align*}
Using again the upper estimate in
\eqref{Ineq2} (or \eqref{S2E2}) as well as
the fact that in the domain of integration $r$ is close to $R$ we obtain%
\[
\int_{\left[  1,\delta R\right]  }d\rho\int_{\left(  R-\rho,R\right]  }%
dr\int_{\Delta^{d-1}}F\left(  \rho,\sigma\right)  \rho^{d-1}d\tau\left(
\sigma\right)  \int_{\Delta^{d-1}}r^{d-1}F\left(  r,\theta\right)
d\tau\left(  \theta\right)  \geq\frac{C_{1}}{R^{\gamma+1}}%
\]
where $C_{1}>0$ depends on $J$ but is independent of $R$ and $\delta$.
We then consider the measure
\[
h\left(  r\right)  =r^{d-1}\int_{\Delta^{d-1}}F\left(  r,\theta\right)
d\tau\left(  \theta\right)\,,
\]
which belongs to $\mathcal{M}_+(\R_*)$, by Fubini's theorem. We can conclude that
for all $R\ge R_\delta$
\[
\int_{\left[  1,\delta R\right]  }h\left(  \rho\right)  d\rho\left[
\int_{\left(  R-\rho,R\right]  }h\left(  r\right)  dr\right]  \geq\frac{C_{1}}{R^{\gamma+1}}. 
\]

The support of $h$ lies in $[1,\infty)$ and $h\ne 0$. 
Since $
 \int_{\R_*\times \Delta^{d-1}} r^{d-1+\gamma} F(r,\theta)dr d\tau(\theta)
<\infty$, we have $ \int_{\R_*} r^{\gamma} h(r)dr <\infty$.
Here $\gamma\ge 1$, and we may conclude 
that also $ \int_{\R_*} h(r)dr <\infty$.  Therefore, we can define a right-continuous, non-negative and non-increasing function  $W$ by
\[
W\left(  R\right)  =\int_{\left( R,\infty\right)  }h\left(  r\right)  dr\,,
\qquad R>0\,,
\]
and rewrite the earlier bound as 
\[
\int_{\left[  1,\delta R\right]  }h\left(  \rho\right)  \left[  W\left(  R-\rho\right) 
-W\left(  R\right)  \right]  d\rho\geq\frac{C_{1}
}{R^{\gamma+1}}%
\]
for $R\ge R_\delta$. Applying Lemma \ref{lem:F+estimate}, we obtain%
\[
W\left(  R\right)  \geq\frac{B}{R^{\gamma}}%
\]
for all $R$ large enough and with a constant $B>0$. 
Then, for any $R$ sufficiently large we have
$$\int_{[R,\infty)} \rho^\gamma h(\rho)d\rho \geq R^{\gamma} W(R)\geq B>0,$$
but this contradicts (\ref{S8E8}) and the result follows. This concludes the proof of item $(i)$. 
\medskip

We now prove item $(ii)$ concerning the non-existence of stationary injection solutions for the discrete model. 
We observe that the result is a direct corollary of  item $(i)$.  Namely, if $\left\{  n_{\alpha}\right\}  _{\alpha\in\mathbb{N}%
_{*}^{d} }$ solves (\ref{S1E5}), then 
$f\left(  \cdot\right)  =\sum_{\alpha
}n_{\alpha}\delta\left(  \cdot-\alpha\right)  $ and $\eta\left(  \cdot\right)
=\sum_{\alpha}s_{\alpha}\delta\left(  \cdot-\alpha\right) $ solves
(\ref{eq:contStat}) which leads to a contradiction.
\end{proofof}

\bigskip

Finally we show the non-existence of nontrivial constant flux solutions stated in Theorem  \ref{NonexFluxSol}. \medskip

\begin{proofof}[Proof of Theorem \ref{NonexFluxSol}]
The proof is similar to that of Theorem
\ref{NonexistCont}, except that here we do not even  need to assume that $f$ solves 
(\ref{WeakForm1}).

We begin with the assumption that $f$ is a nontrivial constant flux solution, and hence $A$ defined by (\ref{S1E8}) satisfies $A_j(R)=(J_0)_j$, for all $R>0$.  Summing over $j$ in 
(\ref{S1E8}), we find that for all $R>0$
\begin{equation}
|J_0| =  \frac{1}{d}
\int_{(0,R]\times \Delta^{d-1}}dr d\tau\left(  \theta\right)\int_{(R-r,\infty)\times \Delta^{d-1}} d\rho d\tau\left(  \sigma\right)
r^{d}\rho^{d-1}
F\left(  r,\theta\right)  F\left(  \rho,\sigma\right) G\left(
r,\rho;\theta,\sigma\right) \,.\label{S1E8sum}
\end{equation}
Since we assume that the measure $f$ is nontrivial, here $|J_0|>0$.
Therefore, the previous equality in (\ref{S8E7}) is satisfied (the value of $L>0$ can now be chosen arbitrarily), and 
we can follow the remaining steps in the proof of Theorem
\ref{NonexistCont} to obtain a contradiction using the moment bound (\ref{S1E4}) and the assumption $\gamma\ge 1$.
\end{proofof}

 \bigskip

%%%%%%%%%%%%%%%%%%%%%%%%%%%%%%%%%%%

\section{Upper and lower estimates for the solutions}\label{sec:7}

\bigskip

Until now, we have focused on the question of existence of stationary solutions in the three cases of interest.
We now consider those kernels which can have a stationary solution.  Although no uniqueness of these solutions is claimed here, we prove in this subsection that all such solutions satisfy certain powerlaw growth bounds.
We collect the statement for the continuous stationary injection solutions in Proposition \ref{UppEstCont} and the corresponding results for the discrete coagulation equation in Proposition \ref{UppEstDisc}.
These results generalize similar upper and lower estimates proven 
in \cite{FLNV} to the multicomponent case and the larger class of kernels.

\begin{proposition}
\label{UppEstCont}
Let us assume that $K$ is a continuous symmetric function satisfying
(\ref{Ineq2}), (\ref{Ineq3}) with $\gamma+2p<1.$ Assume also that 
$\eta\in\mathcal{M}_{+,b}\left(  \mathbb{R}_{*}^{d}\right)$
has its support in  the set $\defset{x\in \R^d_*}{1\le |x|\le L}$ for some $L>1$. We denote by  
$\vert J_{0}\vert $ the total mass injection rate, where $J_{0} =\int
_{\mathbb{R}_{*}^{d}} x\eta\left(  dx\right)\in \R^d_*$. 
Let $f\in\mathcal{M}
_{+}\left(  \mathbb{R}_{*}^{d}\right)$ be a stationary injection solution  to (\ref{eq:contStat}) in the sense
of Definition \ref{StatInjDef}.
There exist constants $C_{1}, C_{2}>0$ and $b\in\left(  0,1\right)  $
depending only on $\gamma,\ p,\ d$ and the constants $c_{1},c_{2}$ in
(\ref{Ineq2}) such that the following estimates hold with $\xi = \frac{L}{b}$ 
\begin{align}
\frac{1}{z}\int_{\frac{z}{2}\leq\left\vert x\right\vert \leq z}f\left(
dx\right)   &  \leq\frac{C_{1}\sqrt{\vert J_{0}\vert }}{z^{\frac{3+\gamma}{2}}%
}\ \ \ \text{for all }z>0\,,\label{S3E4}\\
\frac{1}{z}\int_{bz\leq\left\vert x\right\vert \leq z}f\left(  dx\right)   &
\geq\frac{C_{2}\sqrt{\vert J_{0}\vert }}{z^{\frac{3+\gamma}{2}}}\ \ \ \text{for all }%
z> \xi \, .\label{S3E5} 
\end{align}

Alternatively, let us assume  that $f$ is a nontrivial constant 
flux solution (cf.~Definition \ref{DefConstFlux}) 
with flux $J_0$. Then
there exist constants $C_{1}, C_{2}>0$ and $b\in\left(  0,1\right)  $
depending only on $\gamma,\ p,\ d$ and the constants $c_{1},c_{2}$ in
(\ref{Ineq2}) such that \eqref{S3E4} and \eqref{S3E5} hold with $\xi = 0$.
\end{proposition}

\begin{remark}
We observe that $\vert J_{0}\vert $ is the total injection rate which means that it 
includes all of the possible monomer types.
\end{remark}

\begin{proof}
Due to Lemma \ref{EquivDefin} as well as the fact that the estimates
(\ref{S3E4}), (\ref{S3E5}) are invariant under the transformation $\left(
f\left(  x\right)  ,\gamma,p\right)  \rightarrow\left(  \frac{f\left(
x\right)  }{\left\vert x\right\vert ^{p}},\gamma+2p,0\right)$, it is
sufficient to prove the Proposition for $p=0$ and $\gamma<1$.

Suppose first that $f$ is an appropriate stationary injection solution, in particular, it satisfies
(\ref{WeakForm1}).
We define $F$ by means of (\ref{G2}) and $G$ by means of (\ref{G1}). Then,
arguing as in the derivation of (\ref{FluxTruncMulti}) in 
the proof of Theorem \ref{ThmExistCont} (i.e.~using the test
function $\varphi\left(x\right)  =|x|\chi_{\delta}\left(  |x|\right)  $ and
taking the limit $\delta\rightarrow0$), we obtain 
\begin{align}\label{FluxMultAllz}  
&   \frac{1}{d}\int_{\left(  0,z\right]  }r^{d}dr\int_{\left(  z-r,\infty\right)  }%
\rho^{d-1}d\rho
\int_{\Delta^{d-1}} d\tau\left(  \theta\right)  \int_{\Delta^{d-1}} d\tau\left(  \sigma\right)  G\left(  r,\rho;\theta,\sigma\right)  F\left(  r,\theta\right)  F\left(  \rho,\sigma\right)
\nonumber \\
& \quad = \int_{\left\{ 0< \left\vert x\right\vert \leq z\right\}  }\left\vert
x\right\vert \eta\left(  dx\right) \,, \qquad \text{for any }z>0\,.
\end{align}
The $(r,\rho)$-integration can also be rearranged using Fubini's theorem to occur over the set
\begin{equation}
\Omega_{z}:=\{(r,\rho)\in\mathbb{R}_{+}^{2}:0<r\leq z,\ \rho
>z-r\}\,, \qquad z>0 . \label{eq:Omegaz}%
\end{equation}
In particular, we find for $z$ larger than the support of the source that
\begin{align}\label{FluxMult}  
&   \frac{1}{d}\int_{\Omega_{z}}dr d\rho \, r^{d} \rho^{d-1}
\int_{\Delta^{d-1}} d\tau\left(  \theta\right)  \int_{\Delta^{d-1}} d\tau\left(  \sigma\right)  G\left(  r,\rho;\theta,\sigma\right)  F\left(  r,\theta\right)  F\left(  \rho,\sigma\right)
= |J_0| \,,\ z\ge L\,.
\end{align}

In the second case, if $f$ is a nontrivial constant 
flux solution as in Definition \ref{DefConstFlux},
by summing over $j$ in (\ref{S1E8}) we find that (\ref{FluxMult}) holds for any choice of $L>0$.
In particular, we may now conclude that, for all $z>0$ and for either of two cases of a solution $f$,
we have an upper bound
\begin{align}\label{FluxBoundMult} 
&   \frac{1}{d}\int_{\Omega_{z}}dr d\rho \, r^{d} \rho^{d-1}
\int_{\Delta^{d-1}} d\tau\left(  \theta\right)  \int_{\Delta^{d-1}} d\tau\left(  \sigma\right)
G\left(  r,\rho;\theta,\sigma\right)  F\left(  r,\theta\right)  F\left(  \rho,\sigma\right) 
\le |J_0| \,.
\end{align}
First, let us recall that, since $p=0$,
we have 
\begin{equation}
c_{1}\left(  r+\rho\right)  ^{\gamma}\leq G\left(  r,\rho,\theta
,\sigma\right)  \leq c_{2}\left(  r+\rho\right)  ^{\gamma}\,.\label{IneqG}%
\end{equation}
The integration goes over 
$\Omega_z$ which contains the set $[2z/3,z]^{2}$.  By positivity of the integrand and using the 
lower bound in (\ref{IneqG}), we thus obtain an estimate
\begin{align}
& \frac{2}{3 d} c_1 z^{\gamma+1} \left( 
\int_{[2z/3,z]} dr \, r^{d-1}
\int_{\Delta^{d-1}} d\tau\left(  \theta\right)  F(  r,\theta) \right)^2
\le |J_0| \,, \qquad \text{for any }z>0 \,.
\end{align}
Since $[z/2,z]=[z/2,3 z/4]\cup [2z/3,z]$, we find that
(\ref{S3E4}) holds with a choice of  $C_1>0$ which depends only on $c_1$, $\gamma$, and $p$.  This concludes the proof of the upper estimate, for both of the types of solutions $f$.

For the lower bound, we recall that (\ref{FluxMult}) holds for either of the types of solutions $f$, using an arbitrary $L>0$ if $f$ is a flux solution.
We now prove that the main contribution to the integral in (\ref{FluxMult}) is
due to the regions where $r$ and $\rho$ are comparable. To this end, for each
$\delta\in (0,1)$ we partition $(0,\infty)^{2}
=\Sigma_{1}(\delta)\cup\Sigma_{2}(\delta)\cup\Sigma_{3}(\delta)$ where%
\begin{align}
&\Sigma_{1}(\delta)=\{(r,\rho)\in (0,\infty)^{2}: \rho>r/\delta
\}\,,\nonumber \\& 
\Sigma_{2}(\delta)=\{(r,\rho)\in (0,\infty)^{2}:\delta
r\leq\rho\leq r/\delta\}\,,\nonumber \\& 
\Sigma_{3}(\delta)=\{(r,\rho)\in (0,\infty)^{2}:\rho<\delta r\}\,, \label{SigmaDelta}%
\end{align}
We define also the related functions%
\begin{equation}
J_{j}(z,\delta)= \frac{1}{d}\int_{\Omega_{z}\cap\Sigma_{j}(\delta)}drd\rho
\int_{\Delta^{d-1}} d\tau\left(  \theta\right)  \int_{\Delta^{d-1}} d\tau\left(  \sigma\right) r^{d}\rho^{d-1} 
G\left(  r,\rho;\theta,\sigma\right)  F\left(  r,\theta\right)  F\left(  \rho,\sigma\right)
\label{JiFunctions}%
\end{equation}
for $z>0$, $\delta\in (0,1)$ and $j=1,2,3$. 
We now claim that for each $\varepsilon>0$ there exists
$\delta_{\vep}$ depending only on $\varepsilon$, as well as on $\gamma,\ p,\ d$
and the constants $c_{1},c_{2}$, such that, if $\delta<\delta_\vep$ we have%
\begin{equation}
\sup_{z>0}J_{1}(z,\delta)\leq\varepsilon \vert J_{0}\vert \label{eq:flux_sigma1}%
\end{equation}
and
\begin{equation}
\frac{1}{R}\int_{[R,2R]}J_{3}(z,\delta)dz\leq\varepsilon
\vert J_{0}\vert \quad \text{for each }R>0\,.\label{eq:flux_sigma3}%
\end{equation}

We prove first (\ref{eq:flux_sigma1}).  By (\ref{IneqG}), if $\delta<1$ and $\left(  r,\rho\right)  \in\Sigma_{1}(\delta)$, we have
$G\left(  r,\rho;\theta,\sigma\right)  \leq 2^{|\gamma|}c_{2}\rho^{\gamma},$
and, therefore,
\[
J_{1}(z,\delta)\leq  \frac{1}{d}2^{|\gamma|}c_{2}
\int_{\Omega_{z}\cap\Sigma_{1}(\delta)}drd\rho
\int_{\Delta^{d-1}} d\tau\left(  \theta\right)  \int_{\Delta^{d-1}} d\tau\left(  \sigma\right) r^{d}\rho^{d-1+\gamma} 
 F\left(  r,\theta\right)  F\left(  \rho,\sigma\right)
\,.
\]
We have already proven a powerlaw upper bound in 
(\ref{S3E4}), and this allows further bounding similar integrals via 
Lemma \ref{lem:bound}.  In particular, we may conclude that for any $q\in \R$ there is a constant $C$, which depends only on $q$, $\gamma$, $d$, and $C_1$, such that if $a>0$ and $R\ge  2a$, we have
\begin{align}\label{eq:Fpolqbound}
  \int_{[a,R]} dr  \int_{\Delta^{d-1}} d\tau\left(  \theta\right) r^{d-1+q} F(r,\theta)
  \le C  \sqrt{\vert J_{0}\vert }
  \int_a^R r^{q-\frac{\gamma+3}{2}} dr   \,. 
\end{align}
Now $\Omega_{z}\cap\Sigma_{1}(\delta)$ is contained in
$\left\{(r,\rho):  0< r \le z,\ \rho\ge \frac{r}{\delta},\, \frac{z}{2}\right\}$, 
and we first use the above powerlaw bound to estimate the
$\rho$-integral.  Setting $q=\gamma$ and taking $R\to\infty$, we find that for any $a>0$
\[
\int_{[a,\infty)} d\rho 
 \int_{\Delta^{d-1}} d\tau\left(  \sigma\right) \rho^{d-1+\gamma} 
  F\left(  \rho,\sigma\right)
\le
 C  \sqrt{\vert J_{0}\vert }
  \int_a^\infty \rho^{\gamma-\frac{\gamma+3}{2}} d\rho
  = \frac{2 C}{1-\gamma}\sqrt{\vert J_{0}\vert } a^{-\frac{1-\gamma}{2}}
   \,.
\]
We employ the estimate with $a=\max(\frac{r}{\delta},\frac{z}{2})$, and conclude that
\[
 J_{1}(z,\delta)\leq C  \sqrt{\vert J_{0}\vert } \int_{(0,z]}dr
\int_{\Delta^{d-1}} d\tau\left(  \theta\right)  r^{d}\max\!\left(\frac{r}{\delta},\frac{z}{2}\right)^{-\frac{1-\gamma}{2}}
 F\left(  r,\theta\right) \,.
\]
where the constant $C$ has been adjusted.  We can then apply the first item of Lemma  \ref{lem:bound}
to the interval $[a',z]$
with a weight $\varphi(r) = r \min(\frac{\delta}{r},\frac{2}{z})^{\frac{1-\gamma}{2}}$ and a
bound function $g(r)=C_1\sqrt{d\vert J_{0}\vert } r^{-\frac{1+\gamma}{2}} \min(\frac{2 \delta}{r},\frac{2}{z})^{\frac{1-\gamma}{2}}$.  Taking then $a'\to 0$ shows that there is 
a constant $\bar{C}_{2}$ which depends only on variables allowed in the Proposition and for which
\begin{align}\label{eq:J1delbound}
J_{1}(z,\delta)\leq\bar{C}_{2}\vert J_{0}\vert\left( \int_{0}^{z}\frac{dr}{r^{\frac
{1+\gamma}{2}}}\min\left\{  \frac{1}{z^{\frac{1-\gamma}{2}}},\frac{\delta ^{\frac{1-\gamma}{2}}}{r^{\frac{1-\gamma}{2}}}\right\} + \delta^{\frac{1-\gamma}{2}}\right)\,.
\end{align}
We use the change of variables $r=z\xi$ in the remaining integral which thus becomes equal to
$\int_{0}^{1}d\xi\, \xi^{-\frac{1+\gamma}{2}}
\min\left(  1, \delta  ^{\frac{1-\gamma}{2}}\xi^{-\frac{1-\gamma}{2}}\right)$. The
integral is independent of $z$, well defined for any $\delta>0$, and it converges to zero as $\delta\rightarrow 0$ due to Lebesgue's dominated
convergence theorem.  Thus (\ref{eq:J1delbound}) implies that 
(\ref{eq:flux_sigma1}) holds for all small enough $\delta$.

We next prove (\ref{eq:flux_sigma3}) and now assume that $0<\delta\leq\frac{1}%
{4}$ and choose an arbitrary $R>0$.  Integrating $J_{3}(z,\delta)$ over
$z\in\left[ R,2R\right]  $ and using that $\Omega_{z}\cap\Sigma_{3}\subset\{(r,\rho):0<\rho\leq\delta z,\ z-\rho\leq r\leq z\}$ as well as (\ref{IneqG}), we
obtain%
\begin{align*}
&\int_{\left[  R,2R\right]  }J_{3}(z,\delta)dz \\&  \qquad \leq C_{3}%
\int_{[R,2R]}dz\int_{(0,\delta z]}d\rho\int_{\lbrack z-\rho,z]}r^{\gamma
+1}dr\int_{\Delta^{d-1}}r^{d-1}F\left(  r,\theta\right)  d\tau\left(
\theta\right)  \int_{\Delta^{d-1}}\rho^{d-1}F\left(  \rho,\sigma\right)
d\tau\left(  \sigma\right)  \\
& \qquad \leq\tilde{C}_{3}R^{\gamma+1}\int_{[R,2R]}dz\int_{(0,2\delta R]}d\rho
\int_{\lbrack z-\rho,z]}dr\int_{\Delta^{d-1}}r^{d-1}F\left(  r,\theta\right)
d\tau\left(  \theta\right)  \int_{\Delta^{d-1}}\rho^{d-1}F\left(  \rho
,\sigma\right)  d\tau\left(  \sigma\right)
\end{align*}
where we use that $(0,\delta z]\subset(0,2\delta R].$  Here
$\{(r,z):z-\rho\leq r\leq z,\ R\leq z\leq2R\}\subset
\{(r,z):R/2\leq r\leq2R,\ r\leq z\leq r+\rho\}$ if
$0<\rho\leq R/2$, and thus employing Fubini's theorem we obtain%
\begin{align*}
&\int_{\left[  R,2R\right]  }J_{3}(z,\delta)dz \\&\qquad  \leq\tilde{C}%
_{3}R^{\gamma+1}\int_{(0,2\delta R]}d\rho\int_{\lbrack\frac{R}{2},R]}%
dr\int_{[r,r+\rho]}dz\int_{\Delta^{d-1}}r^{d-1}F\left(  r,\theta\right)
d\tau\left(  \theta\right)  \int_{\Delta^{d-1}}\rho^{d-1}F\left(  \rho
,\sigma\right)  d\tau\left(  \sigma\right)  \\
&  \qquad =\tilde{C}_{3}R^{\gamma+1}\int_{(0,2\delta R]}\rho d\rho\int_{\lbrack
\frac{R}{2},R]}dr\int_{\Delta^{d-1}}r^{d-1}F\left(  r,\theta\right)
d\tau\left(  \theta\right)  \int_{\Delta^{d-1}}\rho^{d-1}F\left(  \rho
,\sigma\right)  d\tau\left(  \sigma\right).
\end{align*}
Using here the estimate (\ref{eq:Fpolqbound}) to bound the remaining $r$- and $\rho$-integrals separately, we find that
\[
\int_{\left[   R,2R\right]  }J_{3}(z,\delta)dz\leq C
\vert J_{0}\vert  R^{\gamma+1}\int_{(0,2\delta R]}\frac{\rho d\rho}{\rho^{\frac{3+\gamma
}{2}}}\int_{[\frac{R}{2},R]}\frac{dr}{r^{\frac{3+\gamma}{2}}}\leq C_{4}%
R \vert J_{0}\vert \delta^{\frac{1-\gamma}{2}}\,.%
\] 
This implies (\ref{eq:flux_sigma3}) for all sufficiently small  $\delta$.

In particular, we can now conclude that for any $\vep>0$ there is a $\delta_\vep\in (0,\frac{1}{4}]$ such that for all $0<\delta<\delta_\vep$ both (\ref{eq:flux_sigma1}) and (\ref{eq:flux_sigma3}) hold. 
In order to prove (\ref{S3E5}), we first note that (\ref{FluxMult}) and
(\ref{JiFunctions}) imply (using an arbitrary $L>0$ if $f$ is a flux solution)
\begin{align}\label{eq:J0sum}
\vert J_{0}\vert =\sum_{j=1}^{3}J_{j}(z,\delta)\ \ \text{for all }z\geq L\,. 
\end{align}
We fix $\vep$ so that $0<\vep\le \frac{1}{4}$ and choose some $\delta\in (0,\delta_\vep)$.
We then integrate (\ref{eq:J0sum})  over $z\in\left[  R,2R\right]  $ with $R\ge L$, 
and use (\ref{eq:flux_sigma1}) and (\ref{eq:flux_sigma3}) to conclude that
\[
\frac{\vert J_{0}\vert R}{2}\leq \vert J_{0}\vert R\left(  1-2\varepsilon\right)  \leq\int_{\lbrack
R,2R]}J_{2}(z,\delta)dz\,.
\]

A simple geometrical argument shows that there exists a constant $b$, $0<b<1$, depending only on the choice of  $\delta$ such that
$\underset{z \in [R,2R]}{\bigcup} (\Omega_z \cap \Sigma_2(\delta)) \subset (\sqrt{b}R, R/\sqrt{b}]^2 $. 
(For a fixed $z$ and $(r,\rho)\in \Omega_z \cap \Sigma_2(\delta)$ one finds
$(\delta^{-1}+1)^{-1} z<r,\rho\le \delta^{-1} z$; thus for example 
$b = \frac{\delta^2}{4}$ would suffice.) 
Moreover, we then have
$G\left(  r,\rho;\theta,\sigma\right)  \leq C_{5}R^{\gamma}$ for all
$\left(  r,\rho\right)  \in\left(  \sqrt{b}R,R/\sqrt{b}\right]  ^{2}$ where
$C_{5}$ depends only on $b$, $\gamma$ and $c_{2}$. Thus using the definition of
$J_{2}(z,\delta)$ given in (\ref{JiFunctions}) we obtain%
\[
\frac{\vert J_{0}\vert R}{2}\leq C'R^{\gamma+2}\left(  \int_{\left\{  \sqrt{b}%
R\leq\left\vert x\right\vert \leq R/\sqrt{b}\right\}  }f\left(  dx\right)
\right)  ^{2}\ \ \text{for }R\geq L\,.
\]
Therefore,
\[
\frac{1}{R}\int_{\left\{  \sqrt{b}R\leq\left\vert x\right\vert \leq R/\sqrt
{b}\right\}  }f\left(  dx\right)  \geq\frac{\tilde{C}_{1}\sqrt{\vert J_{0}\vert }%
}{R^{\frac{3+\gamma}{2}}}\quad \text{for }R\geq L\,,%
\]
and (\ref{S3E5}) follows, after replacing $R/\sqrt{b}$ by $z$ (note that if $z> L/b$, then $\sqrt{b}z> L/\sqrt{b}>L$).  If $f$ is a constant flux solution, then the result holds for any $L>0$, and thus we may also set $\xi=0$ in (\ref{S3E5}).
\end{proof}

\medskip

\begin{remark}
Notice that for $\gamma> -1$, combining Proposition \ref{UppEstCont} with Lemma \ref{lem:bound} we obtain that the
number of clusters associated to the stationary injection solutions
$\int_{{\mathbb{R}}_{*}} f(dx)$ is finite. Moreover, the following estimates hold:
\[
\frac{ C'_{1}\sqrt{\vert J_{0}\vert }}{R^{(\gamma+ 1)/2}} \leq\int_{\left\{\vert x\vert \geq R\right\}} f(dx) \leq
\frac{ C'_{2}\sqrt{\vert J_{0}\vert }}{R^{(\gamma+ 1)/2}}\,, \qquad\text{for } R > L\,,
\]
where $\vert J_{0}\vert = \int_{\R^d_{*}} \vert x \vert \eta(dx)$ and $0< C'_{1} \leq C'_{2} $.
(The lower bound follows by setting $z=R/b>L/b$ in the Proposition.)
\end{remark}

We can now conclude a result similar to Proposition \ref{UppEstCont} for the
solutions of the discrete problem (\ref{S1E5}).  Although a fairly direct corollary of the previous result,
it is worth recording separately in detail, due to the relevance of discrete coagulation equation in applications.

\begin{proposition}
\label{UppEstDisc}
Suppose that $K$ is symmetric and satisfies (\ref{Ineq1}), (\ref{Ineq3}) with
$\gamma+2p<1.$ Let $\left\{  s_{\alpha}\right\}  _{\alpha\in\mathbb{N}_{\ast
}^{d} }$ be a nonnegative sequence supported on a
finite number of values $\alpha$, and suppose $L_s>1$ is such that $s_\alpha=0$ if $|\alpha|> L_s$.
Suppose that $\left\{  n_{\alpha}\right\}  _{\alpha
\in\mathbb{N}_{*}^{d}  }$ is a stationary injection
solution to (\ref{S1E5}) in the sense of Definition \ref{DefStatInjDis}. We
denote as $\vert J_{0}\vert $ the total injection rate of monomers given by $\vert J_{0}\vert 
=\sum_{\alpha}\left\vert \alpha\right\vert s_{\alpha}.$ Then, there exist
positive constants $C_{1},\ C_{2}$ and $b\in\left(  0,1\right)  $ depending
only on $\gamma,\ p,\ d$ and the constants $c_{1},c_{2}$ in (\ref{Ineq1}) such
that%
\begin{align}
\frac{1}{z}\sum_{\frac{z}{2}\leq\left\vert \alpha\right\vert \leq z}%
n_{\alpha}  &  \leq\frac{C_{1}\sqrt{\vert J_{0}\vert }}{z^{\frac{3+\gamma}{2}}%
}\ \ \ \text{for all }z\geq 1\,,\label{S3E6}\\
\frac{1}{z}\sum_{bz\leq\left\vert \alpha\right\vert \leq z}n_{\alpha}  &
\geq\frac{C_{2}\sqrt{\vert J_{0}\vert }}{z^{\frac{3+\gamma}{2}}}\ \ \ \text{for all }%
z\geq\frac{L_{s}}{b}\,. \label{S3E7}%
\end{align}
\end{proposition}

\begin{proof}
This result is just a Corollary of Proposition \ref{UppEstCont} since
$f\left(  \cdot\right)  =\sum_{\alpha}n_{\alpha}\delta\left(  \cdot
-\alpha\right)  $ and $\eta\left(  \cdot\right)  =\sum_{\alpha}s_{\alpha
}\delta\left(  \cdot-\alpha\right)  $ satisfy all the assumptions in
Proposition \ref{UppEstCont}, and (\ref{S3E4}), (\ref{S3E5}) imply
(\ref{S3E6}), (\ref{S3E7}) respectively.
\end{proof}

%%%%%%%%%%%%%%%%%%%%%%%%%%%%%%%%%%%

 \bigskip

\bigskip

\bigskip

\noindent \textbf{Acknowledgements.}
The authors gratefully acknowledge the support of the Hausdorff Research Institute for Mathematics
(Bonn), through the {\it Junior Trimester Program on Kinetic Theory},
of the CRC 1060 {\it The mathematics of emergent effects} at the University of Bonn funded through the German Science Foundation (DFG), 
 of the {\it Atmospheric Mathematics} (AtMath) collaboration of the Faculty of Science of University of Helsinki, of the ERC Advanced Grant 741487 as well as of  
the Academy of Finland via the {\it Centre of Excellence in Analysis and Dynamics Research} (project No.\ 307333).
The funders had no role in study design, analysis, decision to
publish, or preparation of the manuscript.

\bigskip
\noindent\textbf{Compliance with ethical standards} 
\smallskip

\noindent \textbf{Conflict of interest} The authors declare that they have no conflict of interest.

\bigskip 

 \bigskip
 
 \def\adresse{
\begin{description}

\item[M.~A. Ferreira] 
{Department of Mathematics and Statistics, University of Helsinki,\\ P.O. Box 68, FI-00014 Helsingin yliopisto, Finland \\
E-mail:  \texttt{marina.ferreira@helsinki.fi}}

\item[J. Lukkarinen]{ Department of Mathematics and Statistics, University of Helsinki, \\ P.O. Box 68, FI-00014 Helsingin yliopisto, Finland \\
E-mail: \texttt{jani.lukkarinen@helsinki.fi}}

\item[A. Nota:] { Department of Information Engineering, Computer Science and Mathematics,\\ University of L'Aquila, 67100 L'Aquila, Italy\\
E-mail: \texttt{alessia.nota@univaq.it}}

\item[J.~J.~L. Vel\'azquez] { Institute for Applied Mathematics, University of Bonn, \\ Endenicher Allee 60, D-53115 Bonn, Germany\\
E-mail: \texttt{velazquez@iam.uni-bonn.de}}

\end{description}
}

\adresse
 
 \bigskip

\end{document}